\documentclass[12pt]{amsart}
\usepackage{setspace}
\usepackage{verbatim}

\usepackage{amsfonts,fancyhdr,ifthen,bm}
\usepackage{amscd,amssymb} 
\usepackage[letterpaper,left=1.4in,right=1.4in,top=1.35in,bottom=1.35in]{geometry}
 
\usepackage{times}
\usepackage{graphicx}
\usepackage{color}
\usepackage{epsfig}
\usepackage{mathrsfs}
\usepackage{paralist}
\usepackage{comment}
\usepackage{mathtools}
\usepackage{multirow}
\usepackage{tikz-cd}
\usepackage[foot]{amsaddr}
\usepackage{caption}

\usepackage{enumitem}

\usepackage{appendix}

\usepackage{tikz,ifthen,calc}
\usetikzlibrary{automata,arrows,positioning,calc}

\DeclareMathOperator{\ee}{\mathbf{e}}
\DeclareMathOperator{\xx}{\mathbf{x}}

\DeclareMathOperator{\zz}{\mathbf{z}}

\DeclareMathOperator{\uu}{\mathbf{u}}
\DeclareMathOperator{\vv}{\mathbf{v}}

\newtheorem{thm}{Theorem}[section]

\newtheorem*{thm*}{Theorem}
\newtheorem{prop}[thm]{Proposition}
\newtheorem*{prop*}{Proposition}

\newtheorem*{cor*}{Corollary}

\newtheorem{lem}[thm]{Lemma}
\newtheorem*{lem*}{Lemma}

\newtheorem*{oquest*}{Open Question}

\theoremstyle{remark}
\newtheorem{rmk}[thm]{Remark}

\theoremstyle{remark}
\newtheorem*{rmk*}{Remark}

\theoremstyle{definition}
\newtheorem{defn}[thm]{Definition}

\theoremstyle{definition}
\newtheorem{notat}[thm]{Notation}

\theoremstyle{definition}

\theoremstyle{definition}

\theoremstyle{definition}
\newtheorem*{defn*}{Definition}

\theoremstyle{definition}

\theoremstyle{definition}

\theoremstyle{definition}

\numberwithin{equation}{section}

 \usepackage{chngcntr}
\counterwithin{figure}{section}

\newcommand{\C}{\mathbb{C}}
\newcommand{\Z}{\mathbb{Z}}
\newcommand{\QQ}{\mathbb{Q}}

\newcommand{\R}{\mathbb{R}}

\newcommand{\msW}{\mathscr{W}}

\newcommand{\mcH}{\mathcal{H}}

\newcommand{\Gal}{\textup{Gal}}

\newcommand{\Hom}{\textup{Hom}}

\newcommand{\isoarrow}{\xrightarrow{\,\,\,\sim\,\,\,}}

\newcommand{\mfb}{\mathfrak{b}}
\newcommand{\mfg}{\mathfrak{g}}
\newcommand{\mfh}{\mathfrak{h}}

\DeclareFontFamily{U}{wncy}{}
\DeclareFontShape{U}{wncy}{m}{n}{<->wncyr10}{}
\DeclareSymbolFont{mcy}{U}{wncy}{m}{n}
\DeclareMathSymbol{\Sha}{\mathord}{mcy}{"58}

\newcommand\restr[2]{{
  \left.\kern-\nulldelimiterspace 
  #1 
  \vphantom{\big|} 
  \right|_{#2} 
  }}

\newcommand{\mcC}{\mathcal{C}}

\usepackage{scalerel}

\newcommand{\SL}{\textup{SL}}
\newcommand{\slV}{\mathfrak{sl}(V)}
\renewcommand{\sl}{\mathfrak{sl}}
\newcommand{\so}{\mathfrak{so}}
\newcommand{\mft}{\mathfrak{t}}
\newcommand{\mfn}{\mathfrak{n}}
\newcommand{\mfO}{\mathfrak{O}}
\newcommand{\mfs}{\mathfrak{s}}
\newcommand{\Img}{\textup{Im}}
\newcommand{\Q}{\mathbb{Q}}

\newcommand{\vol}{\textup{vol}}
\newcommand{\mcM}{\mathcal{M}}
\newcommand{\dist}{\textup{dist}}

\begin{document}
\title[Lattice points in thickened parabolas]{Lattice points in thickened parabolas and rational points near hypersurfaces}

\author{Alexander Smith}
\email{asmith@northwestern.edu}
\date{\today}

\maketitle

\begin{abstract}
Among the nondegenerate $C^4$ hypersurfaces $\mcM$ in $\R^n$, we characterize the rational quadrics as the hypersurfaces that are the least well approximated by rational points. Given $\mcM$ other than a rational quadric, we prove a heuristically sharp lower bound for the number of rational points very near $\mcM$, improving the sensitivity of prior results of Beresnevich and Huang.

Our methods are dynamical, and rely on an application of Ratner's theorems to $1$-parameter unipotent subgroups $\{u_t \,:\, t \in \R\}$ of $\SL_n(\R)$ such that $u_1 - \text{Id}$ has rank $2$. As part of our work, we study the algebraic subgroups of $\SL_n(\Q)$ whose collection of real points can contain such a subgroup.
\end{abstract}

\section{Introduction}
\subsection{Rational points near hypersurfaces}
The basic problem we start with is to count the number of rational points near a manifold in $\R^n$. This is a problem that has seen a tremendous amount of progress over the last twenty years, starting with lower-bound work by Beresnevich, Dickinson, and Velani \cite{Bere07} for planar curves and Beresnevich \cite{Bere12} for arbitrary submanifolds of $\R^n$. These authors proved their estimates by applying results from homogeneous dynamics, which will also be the approach of this paper.

The complementary upper bound for planar curves was found by Vaughan and Velani \cite{Vaughan06}, who refined an approach using exponential sums due to Huxley \cite{Huxley96}. This approach was extended to give asymptotics for rational points near hypersurfaces by Huang \cite{Huan20}, and then for higher codimension manifolds by Schindler and Yamagishi \cite{Schind22} and Srivastava \cite{Sriv25}. For manifolds other than hypersurfaces obeying certain curvature conditions, these results produce rational points nearer to the manifold than Beresnevich's work.  Our goal is to do the same for hypersurfaces.

For us, a hypersurface $\mcM$ will be a compact connected  codimension-$1$ submanifold of $\R^n$ with boundary for some $n \ge 2$. For $\epsilon > 0$, we define the thickened hypersurface
\[\mathcal{M}_{\epsilon} := \left\{x \in \R^n\,:\,\, \text{dist}(x, \mathcal{M}) \le \epsilon\right\}\]
and, for $Q > 1$, we  define the set of rational points
\[X(\QQ^n, Q) := \{\xx/q \in \QQ^n \,:\,\, (q, \xx) \in \Z^{n+1}\,\text{ and }\, 1 \le q \le Q\}.\]
We are interested in finding good lower bounds for the sizes of the intersections
\[X(\QQ^n, Q) \cap \mcM_{\epsilon}\]
for $\epsilon$ as small as we can manage. We are following in the footsteps of Beresnevich \cite{Bere12}, who showed that, if $\mcM$ is an analytic nondegenerate hypersurface, then there are positive parameters $C, c$ depending on $\mcM$ such that, for all sufficiently large $Q$,
\[\#\big(X(\QQ^n, Q) \cap \mcM_{\epsilon}\big) \ge c Q^{n-1} \quad\text{for}\quad \epsilon = CQ^{-2}.\]
Beresnevich notes that his result cannot hold for substantially smaller $\epsilon$. After all, if $\mcM$ is a rational quadric hypersurface containing no rational points, then there is some positive $c > 0$ such that $X(\QQ^n, Q)$ does not meet $\mcM_{c Q^{-2}}$ for any $Q >1$. For example, if $\mcM$ is the circle of radius $\sqrt{3}$ centered at the origin in $\R^2$, then
\[X(\QQ^n, Q) \cap \mcM_{c Q^{-2}} = \emptyset \quad\text{for all }\, c < 1 \,\text{ and }\, Q > 1.\]
However, if we exempt hyperplanes and rational quadric hypersurfaces, we can improve on Beresnevich's result.
\begin{thm}
\label{thm:hypersurface_qual}
Take $\mcM$ to be a $C^4$ hypersurface in $\R^n$. We assume that $\mcM$ is not contained in any hyperplane. 

Given $Q > 1$, define $\delta_{\mcM}(Q)$ to be the minimal $\delta > 0$ such that
\[X(\Q^n, Q) \cap \left(\mcM_{\epsilon} \backslash \mcM \right) \,\ne\, \emptyset \quad\text{for }\, \epsilon = \delta Q^{-2}.\]

Then exactly one of the following holds:
\begin{enumerate}
\item There is a nonzero integral quadratic polynomial on $\R^{n}$ that vanishes on $\mcM$.
\item We have
\[\limsup_{Q \to \infty} \delta_{\mcM}(Q) = 0.\]
\end{enumerate}
\end{thm}
In the case $n = 2$, this answers a question posed by Beresnevich and Kleinbock  \cite[Problem 5]{KleBer19}. 

This theorem is ineffective; indeed, by taking $\mcM$ to be an irrational quadric hypersurface very well approximated by rational quadric hypersurfaces, we may force $\delta_{\mcM}$ to tend to $0$ arbitrarily slowly. However, for planar curves other than conics, we do get an effective statement.

\begin{thm}
\label{thm:planarcurves}
There is an absolute $\kappa > 0$ so we have the following:

Take $\mcM$ to be a $C^5$ curve in $\R^2$. We assume that there is no nonzero quadratic polynomial that vanishes on $\mcM$.

Then, for all sufficiently large $Q$,
\[X(\Q^n, Q) \cap \left(\mcM_{\epsilon} \backslash \mcM \right) \ne \emptyset \quad\text{for }\, \epsilon = Q^{-2 - \kappa}.\]
\end{thm}

We will also give heuristically sharp lower bounds for the number of rational points near a given hypersurface $\mcM$. Here, our heuristic is the volume-based one also considered by Beresnevich \cite{Bere12} and Huang \cite{Huan20}. Note that the primitive integer points in the cone
\[\mcC_Q(\mcM_{\epsilon}) :=  \left\{(q, \xx) \in \R^{n+1}\,:\,\, 0 < q \le Q \text{ and } \xx/q \in \mathcal{M}_{\epsilon}\right\}\]
 are in bijection with the points in $X(\QQ^n, Q) \cap \mcM_{\epsilon}$. The number of primitive integer points in a large ball in $\R^{n+1}$ is roughly $\zeta(n + 1)^{-1}$ times the volume of the ball, so we might heuristically predict
\begin{equation}
\label{eq:heuristic}
\#\big(X(\QQ^n, Q) \cap \mcM_{\epsilon}\big) \sim \zeta(n+1)^{-1}  \, \vol\, \mcC_Q(\mcM_{\epsilon})  = \frac{2 \epsilon Q^{n+1}}{(n+1) \zeta(n+1)} \textup{vol}\, \mathcal{M},
\end{equation}
where $\vol\, \mcM$ is defined in terms of the usual Riemannian metric on the submanifold $\mcM$. For $\epsilon$ of the form $Q^{-2 + \kappa}$ with $\kappa > 0$, this heuristic was proved to be true for any hypersurface satisfying a certain curvature condition by Huang \cite[Theorem 3]{Huan20}.

The same rational quadrics as before show that Huang's result does not hold for general $\mcM$ if $\epsilon$ is on the order of $Q^{-2}$. But if we exempt the rational quadric hypersurfaces, we can prove Huang's lower bound for $\epsilon$ this small.

\begin{thm}
\label{thm:main}
Take $\mathcal{M}$ to be a $C^4$ hypersurface in $\R^n$. Take $\mathcal{M}_{\textup{flat}}$ to be the set of $x$ in $\mathcal{M}$ such that the curvature tensor of $\mathcal{M}$ is $0$ at $x$, and take $\mathcal{M}_{\textup{quad}}$ to be the set of $x$ in $\mathcal{M}$ such that some nonzero rational quadratic polynomial $P: \R^n \to \R$ vanishes at $x$. Then, for any $\delta > 0$, there is some $Q_0 > 0$ so that, for any $Q > Q_0$ and $\epsilon$ in $\left[\delta Q^{-2}, Q^{-1}\right]$,
\[\#\big(X(\QQ^n, Q) \,\cap\, \mathcal{M}_{\epsilon}\big) \,\ge\, (1 - \delta)\frac{2 \epsilon Q^{n+1}}{(n+1) \zeta(n+1)} \textup{vol}\, \mathcal{M} \backslash (\mathcal{M}_{\textup{flat}} \cup \mathcal{M}_{\textup{quad}}) .\]
\end{thm}
A variant of this result that applies the effective results in \cite{Lind25eff} is the following.

\begin{thm}
\label{thm:main_eff}
There is an absolute $\kappa > 0$ so we have the following:

Take $\mathcal{M}$ to be a $C^5$ curve in $\R^2$.  Take $\mathcal{M}_0$ to be the set of $x \in \mathcal{M}$ such that the curvature of $\mathcal{M}$ at $x$ is nonzero, and such that the osculating conic to $\mathcal{M}$ at $x$ has order of contact exactly four.

Then, for any $\delta > 0$, there is some $Q_0 > 0$ so that, for any $Q > Q_0$ and $\epsilon$ in $\left[Q^{-2 - \kappa}, Q^{-1}\right]$,
\[\#\big(X(\QQ^n, Q) \,\cap\, \mathcal{M}_{\epsilon}\big) \,\ge\, (1 - \delta)\,\frac{2 \epsilon Q^{3}}{3\cdot\zeta(3)} \,\textup{vol}\, \mathcal{M}_0.\]
\end{thm}

As in Beresnevich's work, our approach to proving these theorems is to decompose the cone $\mcC_Q(\mcM_{\epsilon})$ into simpler solids where we can prove lower bounds for the lattice point count. In Beresnevich's work, this cone is approximately decomposed into a union of parallelepipeds of volume $\gg 1$. Applying dynamical work of Kleinbock and Margulis \cite{Klein98} shows that most of these parallelepipeds have large lattice widths, and the flatness theorem \cite{Bana99} then gives a lower bound for the number of lattice points in the cone.

For our work, we instead cover this cone with what we will call thickened parabolas, non-convex sets that may be written as unions of parabolas. The central result of this paper gives a criterion for these sets to contain integer points. 

\subsection{Thickened parabolas}
\begin{defn}
\label{defn:1par}
Choose an integer $n \ge 3$, and take $U = \{u_t:\, t \in \R\}$ to be a $1$-parameter unipotent subgroup of $\SL_n(\R)$. Taking $u = u_1$, we will assume that
\[\textup{rank}(u - \text{Id}) = 2 \quad\text{and}\quad (u - \text{Id})^2 \ne 0.\]

Given $x \in \R^n$ outside the kernel of $(u - \text{Id})^2$, we see that $Ux$ is a parabola. With this in mind, given a nonempty open set $B$ in $\R^n$, we refer to $UB = \{ub\,:\,\, u \in U, b \in B\}$ as a \emph{thickened parabola}.
\end{defn}

We are interested in $UB \cap \Z^n$, the set of integer points in this thickened parabola. This set may be empty. After all, given a quadratic form $P$ on $\R^n$ preserved by $u$, we have
\[P(u_t B) = P(B) \quad \text{for all } t.\]
In the case that $P$ is an integral quadratic form, $P(\Z^n)$ is a subset of $\Z$, and $UB \cap \Z^n$ is empty so long as $P(B)$ does not meet this set.

Outside of this case, thickened parabolas contain integer points.
\begin{thm}
\label{thm:parabola}
Take $U$ to be the $1$-parameter unipotent subgroup of $\SL_n(\R)$ considered above. Suppose no nonzero rational quadratic form is preserved by $U$. Then, for every nonempty open set $B$ of $\R^n$, the thickened parabola $UB$ contains infinitely many points in $\Z^n$.
\end{thm}
Like with Theorem \ref{thm:main}, we can give a heuristically sharp lower bound for the number of integer points in a segment of a thickened parabola so long as $U$ does not preserve a nonzero integral quadratic form. For our application to hypersurfaces, we will need this result to be uniform over a family of thickened parabolas. We give this result as Theorem \ref{thm:uniform_parabola}.

Theorems \ref{thm:parabola} and \ref{thm:uniform_parabola} are proved as a consequence of Ratner's theorems on unipotent flows \cite{Ratn91} and the uniform variants of these theorems due to Dani and Margulis \cite{DaMa93}. Assuming Ratner's theorems, it is straightforward to show that the next theorem implies Theorem \ref{thm:parabola}.

\begin{thm}
\label{thm:unRatnered}

With $U$ as above, take $H$ to be the minimal Zariski closed subgroup of $\SL_n(\QQ)$ such that $H(\R)$ contains $U$. Suppose no nonzero rational quadratic form is preserved by $H$. Then there is a union $Z$ of finitely many proper subspaces of $\R^n$ such that $H(\R)^0$ acts transitively on $\R^n \backslash Z$.
\end{thm}

This theorem is proved Lie theoretically. More specifically, the requirement that $H(\R)$ contains $U$ puts a heavy restriction on the forms that the semisimplification of $H$ can take; see Theorem \ref{thm:semisimple}. After proving this, the theorem is proved by considering the radical of $H(\R)^0$. This Lie-theoretic approach is most closely related to previous work of Dani and Margulis for $\R^3$ \cite{DaMa90} and Gorodnik for $\R^4$ \cite{Goro04}.

\subsection{Diophantine approximation for systems of forms}
Besides the applications to Diophantine approximation on hypersurfaces, our work has applications to Oppenheim's conjecture with linear constraints. Given $n \ge 3$, take $P: \R^n \to \R$ to be an indefinite quadratic form of rank $\ge 3$ that is not proportional to a rational quadratic form. Oppenheim's conjecture, as proved by Margulis \cite{Marg89}, gives that $P(\Z^n)$ is dense in $\R$.

Now choose linearly independent linear functions $A_1, \dots, A_{n-2}$ on $\R^n$. The equation $P  = 0$ defines a hypersurface $X$ in $\mathbb{P}^{n-1}_{\R}$, and $A_1 = \dots = A_{n-2} = 0$ defines a line $L$ in this projective space. We assume that $L$ is tangent to $X$ at a nonsingular point $x_0$ and otherwise does not intersect $X$.

\begin{thm}
\label{thm:localOp}
Suppose that no nonzero quadratic polynomial of the form
\[aP + \sum_{i \le j \le {n-2}} a_{ij} A_iA_j\]
 is rational. Then, given any nonempty open set $B$ of $\R^{n-1}$, there is some $Q_0 > 0$ so that, for any $Q > Q_0$, there is some primitive $\xx = (x_1, \dots, x_{n}) \in \left(\Z^{n}\right)_{\textup{prim}}$ so that
\[\big(P(\xx),\, A_1(\xx), \,\dots, \,A_{n-2}(\xx)\big) \in B\quad\text{and} \quad|x_1|, \dots, |x_n| \le Q.\]
\end{thm}
Here, we have used the standard metric on $S^{n-1}$ to define the distance function on $\mathbb{P}^{n-1}_{\R}$.

 This result was proved in the case $n = 3$ by Dani and Margulis \cite{DaMa90}, in a slightly weaker form in the case $n = 4$ by Gorodnik \cite{Goro04}, and in the case of general $n$ under a genericity hypothesis by Dani \cite{Dani08}.

\subsection{The layout of this paper}
In Section \ref{sec:DaMa}, we prove a uniform lower bound for the primitive integer points in a thickened parabola conditional on the Lie-theoretic result Theorem \ref{thm:unRatnered} by applying Ratner's theorem and dynamical results of Dani and Margulis. 

In Section \ref{sec:hype}, we apply our results on thickened parabolas to find rational points near patches of quadric hypersurfaces. We then prove Theorem \ref{thm:main} in Section \ref{ssec:main} by locally approximating arbitrary hypersurfaces with quadrics. This is followed by the proofs of the other main theorems in Section \ref{ssec:RAG}.

In Sections \ref{sec:semisimple} and \ref{sec:radical}, we finish the paper by giving a proof of Theorem \ref{thm:unRatnered}. This starts by first determining the possible relevant semisimple Lie groups, which we do in Section \ref{sec:semisimple} before handling radicals Section \ref{sec:radical}.

\subsection*{Acknowledgements}
We would like to thank Osama Khalil, who pointed out the relevance of \cite{Lind25eff}. This work also benefited from useful conversations with Victor Beresnevich, William Duke, Tom Gannon, James Leng, Redmond McNamara, Peter Sarnak, and Jeremy Taylor. 

The author served as a Clay Research Fellow for part of the writing of this paper, and would like to thank the Clay Mathematics Institute for their support.

\section{Applying Ratner's theorems}
\label{sec:DaMa}

The goal of this section is to prove Theorem \ref{thm:uniform_parabola}, which gives a lower bound for the number of primitive lattice points in a thickened parabola. Throughout this section, we will assume Theorem \ref{thm:unRatnered}; the proof of this result can be found in Sections \ref{sec:semisimple} and \ref{sec:radical}.

\begin{notat}
\label{notat:uniform_parabola}
Fix some $n \ge 2$. Take $\msW$ to be the set of nilpotent $w$ in $\sl_n(\R)$ such that $w^2 \ne 0$ and $w$ has rank $2$.

Given $w$ in $\msW$, we define the hyperplane
\[L_w = \Img(w)^{\perp} \oplus \Img(w^2) \subseteq \R^n,\]
where the orthogonal complement is taken with respect to the usual inner product structure on $\R^n$. This is defined so that the map $(t, v) \mapsto \exp(tw)v$ defines a homeomorphism
\[\R \times \left(L_w \backslash \ker\, w^2\right)\isoarrow \R^n \backslash \ker w^2.\]

We take $\msW_0$ to be the subset of $w$ in $\msW$ such that $\exp(w)$ preserves no nonzero rational quadratic form.
\end{notat}

\begin{thm}
\label{thm:uniform_parabola}
Choose a compact subset $K$ of $\msW_0$ and $\epsilon > 0$. Then there is some $T_0 > 0$ so we have the following:

Choose $w \in K$ and an open subset $B$ of $L_w$. We assume all vectors in $B$ have norm at most $\epsilon^{-1}$. Take
\[B'  = \left\{z \in B\,:\,\, \dist(z, L_w \backslash B) > \epsilon\right\}.\]
 Then, for $\beta >  1 + \epsilon$, we have
\begin{equation}
\label{eq:uniform_conclusion}
\# \left(U_{[T, \beta T]} B \cap (\Z^n)_{\textup{prim}} \right)\,\ge\, (1 - \epsilon)\cdot  \zeta(n)^{-1} \cdot \vol\, U_{[T, \beta T]} B'.
\end{equation}
\end{thm}
Here, $(\Z^n)_{\text{prim}}$ denotes the primitive points in the lattice $\Z^n$.

\subsection{A version for functions}
Given $\epsilon > 0$, take $\mathscr{F}_{\epsilon}$ to be the set of nonnegative functions $f: \R^n \to \R^{\ge 0}$ whose support is contained in the ball of radius $\epsilon^{-1}$ centered at the origin and which satisfy
\[\left|f(x) - f(y)\right| \le \epsilon^{-1} || x - y|| \quad\text{for all }x, y \in \R^n.\]
An equivalent form of Theorem \ref{thm:uniform_parabola} is the following:

\begin{thm}
\label{thm:uniform_parabola_func}
Choose a compact subset $K$ of $\msW_0$ and $\epsilon > 0$. Then there is some $T_0 > 0$ so we have the following:

Choose $w \in K$, $T > T_0$, $\beta > 1 + \epsilon$, and $f \in \mathscr{F}_{\epsilon}$. Then
\begin{equation}
\label{eq:uniform_parabola_func}
\int_T^{\beta T} \sum_{v \in (\Z^n)_{\textup{prim}}}  f(\exp(tw) v)dt \,\ge\, -\epsilon T + \zeta(n)^{-1} \cdot (\beta - 1) T\cdot  \int_{\R^n} f(x)dx,
\end{equation}
where the measure on $\R^n$ is the standard Euclidean measure.
\end{thm}

We will prove that this theorem implies Theorem \ref{thm:uniform_parabola}. The opposite implication is proved similarly.

\begin{proof}[Proof  that Theorem \ref{thm:uniform_parabola_func} implies Theorem \ref{thm:uniform_parabola}]
We note that it suffices to prove the alternative version of the theorem where \eqref{eq:uniform_conclusion} is swapped out for
\begin{equation}
\label{eq:uniform_conclusion_alt}
\# \left(U_{[T, \beta T]} B \cap (\Z^n)_{\textup{prim}} \right)\,\ge\, - \epsilon T +    \zeta(n)^{-1} \cdot \vol\, U_{[T, \beta T]} B'
\end{equation}
under the condition that $\beta < 2$. For suppose that this alternative version of the theorem holds, and we wish to prove the original version for $(\epsilon, K)$, so that we may produce $T_0$ satisfying the condition of the theorem. We note that the original theorem does not lose strength if we assume $\beta \le 2$, so we make this assumption.

So long as $\epsilon$ is sufficiently small relative to $K$, we claim the $T_0$ produced by the alternative version of the theorem for  $(\epsilon_0 = \epsilon^{n+1}, K)$ suffices for the original version for $(\epsilon, K)$. To show this, suppose we have $B$ satisfying the conditions of Theorem \ref{thm:uniform_parabola}. If $B'$ is empty, the theorem is clearly true. Otherwise, taking
\[B[\epsilon_0] = \left\{z \in B\,:\,\, \dist(z,  L_w\backslash B) > \epsilon_0\right\},\]
we find that $B[\epsilon_0]$ contains $B'$ and some disjoint ball of radius $\epsilon/2$ so long as $\epsilon < 1/2$. So there is some $c > 0$ depending on $K$ such that
\[ \vol\, U_{[T, \beta T]} B[\epsilon_0] \ge c \epsilon^{n-1}(\beta - 1)T +  \vol\, U_{[T, \beta T]} B'.\]
So long as $\epsilon$ is sufficiently small relative to $K$, we then must have
\[- \epsilon_0 T +   \vol\, U_{[T, \beta T]} B[\epsilon_0]  \ge \vol\, U_{[T, \beta T]} B',\]
giving the reduction.

We now prove this alternative version of the theorem. Given $w$ and $B$ satisfying the condition of the theorem and $\delta  < \epsilon$, take
\[B_1 = \{z \in B\,:\,\, \dist(z, \ker(w^2) \cup L_w \backslash B) \ge \delta\},\]
and take
\[Y = \left\{ \exp(tw) z \,:\,\, z \in B_1 \quad\text{and}\quad t \in [\delta, 1 - \delta]\right\}.\]
We now may choose $\epsilon_1  \in (0, \delta)$ depending only on $K$ and $\delta$ such that, if we take
\[Y_0 = \{ x \in \R^n \,:\,\, \dist(x, Y) \le \epsilon_1\},\]
then
\[U_{[0, 1]} B \supseteq Y_0.\]
The function $f: \R^n \to \R^{\ge 0}$ defined by
\[f(x) = \epsilon_1^{-1} \max(0, \epsilon_1 - \dist(x, Y)),\]
then lies in $\mathscr{F}_{\epsilon_1}$ and satisfies
\[\# \left(U_{[T, \beta T]} B \cap (\Z^n)_{\textup{prim}} \right) \ge \int_{T}^{\beta T - 1} \sum_{v \in (\Z^n)_{\textup{prim}}}  f(\exp(tw) v)dt.\]
Applying Theorem \ref{thm:uniform_parabola_func}, we see that the right hand side of this inequality is at least
\[-\epsilon_1 T + \zeta(n)^{-1} \cdot (\beta - 1) T\cdot  \int_{\R^n} f(x)dx\]
so long as $T$ is sufficiently large relative to $K$ and $\epsilon_1$. By the construction of $f$, this is at least
\[ -\epsilon_1 T + \zeta(n)^{-1} \cdot (1 - 2 \delta) \vol U_{[T, \beta T]}B_1.\]
The result follows by taking $\delta$ sufficiently small and $\epsilon_1$ sufficiently small relative to $\delta$.
\end{proof}

\subsection{The proof of Theorem \ref{thm:uniform_parabola_func}}

We record a general result in the style of Dani and Margulis \cite{DaMa93}.
\begin{lem}
\label{lem:DaniMarg}
Take $G$ to be a connected Lie group with Lie algebra $\mfg$, and choose a lattice $\Gamma$ in $G$.  Choose a $1$-parameter unipotent subgroup $U = \{\exp(tw)\,:\,\, t \in \R\}$ of $G$, where $w$ lies in $\mfg$, and take $L$ to be the minimal closed subgroup of $G$ containing $U$ such that $L \cap \Gamma$ is a sublattice of $L$, as exists by \cite[Theorem 2.3]{Shah91}.

Choose a sequence $w_1, w_2, \dots$ of nilpotent elements in $\mfg$ converging to $w$, and choose an unbounded sequence $T_1 < T_2 < \dots$ of positive real numbers.
Then there is a $U$-invariant probability measure $\mu$ on $G/\Gamma$ and a subsequence $k_1 < k_2 < \dots $ such that, for any continuous function $f: G/\Gamma \to \R$ with compact support,
\begin{equation}
\label{eq:weakstar}
\int_{G/\Gamma} f d\mu =\lim_{i \to \infty} \frac{1}{T_{k_i}}\int_{0}^{T_{k_i}} f\big(\exp(tw_{k_i}) \big)dt.
\end{equation}
Furthermore, any $U$-ergodic component of $\mu$ is $gLg^{-1}$-invariant for some $g$ in $G$ such that $gLg^{-1}$ contains $U$.
\end{lem}
\begin{proof}
The existence of the measure $\mu$ follows as in the proof of \cite[Theorem 2]{DaMa93}. We just need to prove the statement about its $U$-ergodic components.

Take $\mcH$ to be the set of closed Lie subgroups $H$ of $G$ such that $H \cap \Gamma$ is a sublattice of $H$ and such that $\text{Ad}(H \cap \Gamma)$ is Zariski dense in $\text{Ad}(H)$. For each $H$ in $\mcH$, take $X(H)$ to be the set of $g$ such  that $gHg^{-1}$ contains $U$. Following \cite[Section 3]{DaMa93}, we may define a continuous linearization map
\[\eta_H: G \to V_H,\]
where $V_H$ is a finite-dimensional real vector space. Inside $V_H$, a Zariski closed subset $A_H$ may be defined so that $\eta_H^{-1}(A_H) = X(H)$ \cite[Proposition 3.2]{DaMa93}.

Suppose that the image of $X(H)$ in $G/\Gamma$ has positive measure under $\mu$. Then there is a compact subset $D$ of $A_H$ so that the image of $\eta_H^{-1}(D)$ in $G/\Gamma$ has positive measure $\delta$. We apply the theorem \cite[Theorem 7.3]{DaMa93} of Dani--Margulis to $E = \eta_H^{-1}(D)$ with $\epsilon = \delta/2$. This produces a sequence $H_1, \dots, H_k$ of groups in $\mcH$ and, for $i \le k$, a compact subset $D_i$ of $A_{H_i}$, that may be used to define
\[E' = \eta_{H_1}^{-1}(D_1) \cup \dots \cup \eta_{H_k}^{-1}(D_k).\]
Following the proof of Dani--Margulis, we see that we may assume that the $H_i$ are all contained in $H$.

Given any open subset $W$ of $V_H$ containing $D$, we see that
\[\text{meas}\left(\left\{t \in [0, T_i]\,:\,\, \exp(tw_i) \in \eta_H^{-1}(W)\Gamma\right\} \right) > \epsilon T_i\]
for all sufficiently large $i$ by \eqref{eq:weakstar}. Then the theorem of Dani--Margulis gives that, if we choose an open subset $W_j$ of $V_{H_j}$ containing $D_j$ for each $j \le k$, there is some $j \le k$ and some $\gamma \in \Gamma$ such that
\[\eta_{H_j}(\gamma) \in W_j.\]
Since the image $\eta_{H_j}(\Gamma)$ is discrete \cite[Theorem 3.4]{DaMa93}, this implies that there is some $\gamma$ in $\Gamma$ and some $j$ such that $\eta_{H_j}(\gamma)$ lies in $A_{H_j}$, so $\gamma$ lies in $X(H_j)$ and
\[\gamma H_j \gamma^{-1} \supseteq U.\]
Then $\gamma H_j \gamma^{-1}$ contains $L$ as well, so $\gamma H \gamma^{-1}$ contains $L$.

We have thus shown that the measure of the image of a set $X(H)$ in $G/\Gamma$ can have positive measure under $\mu$ only if $H$ contains some $\Gamma$-conjugate of $L$. The lemma now follows from \cite[Theorem 2.2]{MoSh95}.
\end{proof}

We now translate Theorem \ref{thm:unRatnered} into a statement about measures.
\begin{lem}
\label{lem:Haar}
With all notation as in Theorem \ref{thm:uniform_parabola_func}, choose $w \in K$, and take $U = \{\exp(tw)\,:\,\, t \in \R\}$. Take $L$ to be the minimal closed subgroup of $\SL_n(\R)$ containing $U$ such that $L \cap \SL_n(\Z)$ is a sublattice of $L$, and take $\mu$ to be an $gLg^{-1}$-invariant probability measure on $\SL_n(\R)/\SL_n(\Z)$ for some $g \in \SL_n(\R)$. Then, for any continuous function $f: \R^n \to \R^{\ge 0}$ with compact support,
\begin{equation}
\label{eq:Riesz}
\int_{\SL_n(\R) /\SL_n (\Z)} \sum_{v \in (\Z^n)_{\textup{prim}}} f(hv) d\mu(h) \,\ge\, \zeta(n)^{-1} \int_{\R^n}  f(x)dx,
\end{equation}
where the measure on $\R^n$ is the usual Euclidean measure.
\end{lem}
\begin{proof}
From \cite[Proposition 3.2]{Shah91}, we know that $L$ is $H(\R)^0$ for $H$ the minimal Zariski closed subgroup of $\SL_n(\Q)$ such that $H(\R)$ contains $U$. From Theorem \ref{thm:unRatnered}, we know that $gLg^{-1}$ acts transitively on a dense open subset $Y$ of $\R^n$.

First suppose that the left hand side of \eqref{eq:Riesz} is infinity for some  $f$ whose support is contained in $Y$, say $f = f_0$. There is then some $x \in Y$ in the support of $f_0$ such that the integral is infinite for any choice of $f$ which is nonzero at $x$. Since $gLg^{-1}$ is transitive on $Y$, the left hand side is infinite for any $f$ which is nonzero at some point in $Y$. This is the set of all nonzero $f$, giving the lemma in this case.

Otherwise, the functional taking $f$ to the left hand side of \eqref{eq:Riesz} for any $f \in C_c(Y)$ is given by a Radon measure on $Y$ by the Riesz representation theorem \cite[Theorem 7.2]{Foll99}. This is a $gLg^{-1}$-invariant measure on the homogeneous space $Y$, which uniquely determines the measure up to scaling \cite[Theorem 2.51]{Foll16}.  Since the restriction of the Euclidean measure on $\R^n$ to $Y$ is preserved by $gLg^{-1}$, we find that
\begin{equation}
\label{eq:Yc}
\int_{\SL_n(\R) /\SL_n(\Z)}\sum_{v \in (\Z^n)_{\textup{prim}}} f(hv) d\mu(h) = \frac{c}{\zeta(n)} \int_{R^n} f(z) dz
\end{equation}
holds for some nonnegative $c$ for any $f$ in $C_c(Y)$.

Fix some nonzero nonnegative $f$ in $C_c(Y)$. Then, given any compact subset $J$ of $\SL_n(\R)$ and any $\delta > 0$, we find there is some $r_0 > 0$ such that
\[\left|\sum_{v \in (\Z^n)^{\textup{prim}}}  f(r^{-1} gv) - \frac{1}{\zeta(n)}\int_{\R^n} f(r^{-1} z)dz\right| \le \delta r^n \quad\text{for all } r> r_0,\,\, g \in J.\]
Given any $\epsilon > 0$,  we may choose $J$ so that the image of $J$ in $\SL_n (\R)/\SL_n(\Z)$ has measure at least $1 - \epsilon/2$. In this way, we find that
\[\int_{\SL_n(\R) /\SL_n(\Z)}\sum_{v \in (\Z^n)_{\textup{prim}}} f(r^{-1}hv) d\mu(h)  \ge \frac{1 - \epsilon}{\zeta(n)} \int_{R^n} f(r^{-1}z) dz\]
for $r$ sufficiently large given $\epsilon$. This implies that $c \ge 1$ in \eqref{eq:Yc}.
\end{proof}

We now have everything we need to prove the theorem.

\begin{proof}[Proof of Theorem \ref{thm:uniform_parabola_func}]
Suppose the result does not hold for a given $K$ and $\epsilon > 0$. This implies that there is a sequence of tuples
\[(w_i,\, T_i, \,\beta_i, \,f_i)\quad\text{ for } i \ge 1\]
satisfying the conditions of the theorem and with $T_i > i$ such that the conclusion of the theorem does not hold.

Since $K$ is compact, some subsequence of the $w_i$ converges to some $w$ in $K$.  Passing to a subsequence of the $w_i$ if necessary and applying Lemma \ref{lem:DaniMarg}, we may assume there is a measure $\mu$ on $\SL_n(\R)/\SL_n (\Z)$ such that
\[\int_{\SL_n (\R)/\SL_n (\Z)} g d\mu = \lim_{i \to \infty}\frac{1}{(\beta_i - 1) T_i} \int^{\beta_i T_i}_{T_i} g(\exp(tw_{i})) dt\]
for all continuous functions $g: \SL^n \R/\SL^n \Z \to \R$ with compact support. Furthermore, defining $L$ as in the lemma, we find that $\mu$ has an ergodic decomposition into $gLg^{-1}$-invariant measures, where $g$ is allowed to vary in $\SL^n \R$.

Passing again to a subsequence, we may assume that the $f_i$ converge to a given continuous $f$ in $\mathscr{F}_{\epsilon}$. By Lemma \ref{lem:Haar}, we have
\[\int_{\SL_n (\R)/\SL_n (\Z)}\sum_{v \in (\Z^n)_{\text{prim}} } f(gv) d\mu(g) \,\ge\, \zeta(n)^{-1} \int_{\R^n} f(x)dx.\]
Take $F_1 \le F_2 \le \dots$ to be a sequence of continuous functions $F_i: \SL_n (\R)/\SL_n (\Z) \to \R^{\ge 0}$ with compact support such that $\lim_{i \to \infty} F_i(x) = 1$ for all $x \in \SL_n (\R)/\SL_n (\Z)$. For any $\delta> 0$, there is some $k \ge 0$ such that for $j \ge k$ we have
\[\int_{\SL_n (\R)/\SL_n (\Z)} \sum_{v \in (\Z^n)_{\text{prim}} }F_j(g) f(gv) d\mu(g) \,\ge\, \frac{1 - \delta}{\zeta(n)} \int_{\R^n} f(x)dx.\]
Then, for $j \ge k$,
\[\lim_{i \to \infty}\frac{1}{(\beta_i - 1) T_i} \int^{\beta_i T_i}_{T_i} \sum_{v \in (\Z^n)_{\text{prim}} }F_j(\exp(tw_i)) f_i(\exp(tw_i)v) dt \ge \frac{1 - \delta}{\zeta(n)} \int_{\R^n} f(x)dx.\]
But, for $\delta$ sufficiently small, this contradicts the assumption that the $(w_i, \, T_i,\, \beta_i, \, f_i)$ do not satisfy \eqref{eq:uniform_parabola_func}. This finishes the proof.
\end{proof}

\subsection{An effective version for $n = 3$}
\label{ssec:eff}
In the case $n = 3$, recent results of Lindenstrauss--Mohammadi--Wang--Yang \cite{Lind25eff} allow us to prove effective lower bounds for the number of points in a thickened parabola.

To start, given positive numbers $C$ and $T_0 > 10$, take $\msW(C, T_0)$ to be the set of $w \in \sl_3(\R)$ such that
\begin{itemize}
\item We have  $||w|| < C$, $||w^2|| > C^{-1}$, and $w^3 = 0$, and
\item For any nonzero integral quadratic form $P: \R^3 \to \R$ whose coefficients are bounded by $T_0$, we have
\begin{equation}
\label{eq:nottooclose}
|P(v)| \ge T_0^{-C},
\end{equation} 
where $v$ is a unit vector in $\Img\, w^2$.
\end{itemize}

\begin{thm}
Given $C > 0$, there is $c_1, C_1 > 0$ so that, for any $T_0 > 10$, any $w \in \msW(C, T_0)$, any $T > T_0^{C_1}$, and any function $f$ in $C^{\infty}_c(\SL_3(\R)/\SL_3(\Z))$, we have
\[\left|\frac{1}{T} \int_0^{T} f(\exp(tw)) dt - \int f d\mu\right| \le T_0^{-c_1}\cdot  \mathcal{S}(f),\]
 where $\mathcal{S}$ is a certain Sobolev norm and $\mu$ is the standard probability measure on $\SL_3(\R)/\SL_3(\Z)$.
\end{thm}

\begin{proof}
Take $\mathcal{H}$ to be the set of connected rational subgroups of $\SL_3(\R)$ whose radical equals their unipotent radical. Given $H$ in $\mathcal{H}$,  define $||\cdot ||$ and $\eta_H: \SL_3(\R) \to \wedge^{\dim H} \sl_3(\R)$ as in \cite[Section 1.2]{Lind24qual}.  Furthermore, take $\mfh(H)$ to be the Lie algebra associated to the maximal subgroup of $N_{\SL_3(\R)}(H)$ generated by unipotent elements.

Then there is an absolute constant $C_2 > 0$ such that, if
\[\max_{t \in [0, S]} ||\eta_H\left(\exp(tw) \right)|| \le ||\eta_H(0)||^{-C_2} S^{1/2}\]
for a given $w \in \msW(C, T_0)$ and $S > 1$, we may conclude that $w$ differs from a nilpotent element in $\mfh(H)(\C)$ by an element of $\sl_3(\C)$ bounded in magnitude by $S^{-1/C_2}$ \cite{Bron88}. So, if $\mfh(H)$ preserves a nonzero integral quadratic form with coefficients bounded by $T_0$, we must have
\[\max_{t \in [0, T]} ||\eta_H\left(\exp(tw) \right)|| \ge T^{1/3}.\]
for all $T$ larger than some $T_0^{C_3}$, where $C_3 > 0$ depends only on $C$.

We now wish to prove that the first condition of \cite[Theorem 1.2]{Lind25eff} holds for our given choice of $w$. From the proof of this theorem, we see that this will follow if we can show that, for some $c_4 > 0$, we have
\begin{equation}
\label{eq:glubb}
\max_{t \in [0, T]} ||\eta_H\left(\exp(tw) \right)|| \ge T^{1/3}
\end{equation}
for $T > T_0^{C_3}$ whenever $H(\R)$ is one of
\begin{itemize}
\item A principal image of $\SL_2(\R)$ in $\SL_3(\R)$, or
\item The unipotent radical of a parabolic subgroup
\end{itemize}
and $||\eta_H(0) || \le T_0^{c_4}$. 

In the second case, the authors show that $H$ may be assumed to contain a conjugate of $\exp(w)$. This implies that it is the unipotent radical of a Borel subgroup.

In both cases, so long as $c_4$ is sufficiently small, we find that we may assume that $\mfh(H)$ preserves a nonzero integral quadratic form whose coefficients are bounded by $T_0$. With our work above, this implies that \eqref{eq:glubb} holds.
\end{proof}

Following the argument from earlier in this section then gives the following:
\begin{thm}
\label{thm:R3_thickened}
Choose $C > 0$, $\delta > 0$, and $T_0 > 10$. Then there is $c_1, C_1 > 0$ determined from $C$ and $\delta$ so that, given any $w \in \mathscr{W}(C, T_0)$, any $T > T_0^{C_1}$, any $\beta > 1 + \delta$, and any ball $B$ in $\{v \in \R^3\,|\,\, |v| \le C\}$  whose radius is at least $T_0^{-c_1}$, we have
\[\#\left( U_{[T, \beta T]} B \cap (\Z^3)_{\textup{prim}}\right) \ge (1 - \delta) \cdot \zeta(3)^{-1} \cdot  \vol\, U_{[T, \beta T]} B.\]
\end{thm}

\section{Local approximation with quadrics}
\label{sec:hype}
The goal of this section is to decompose cones over thickened hypersurfaces into disjoint thickened parabolas with negligible remainder. This will allow us to prove Theorem \ref{thm:main} as a consequence of our uniform estimate for the number of lattice points in a thickened parabola from last section.

To do this work, we first show that local patches of quadric hypersurfaces may be decomposed into thickened parabolas. This work will also allow us to prove the local form of Oppenheim's conjecture given as Theorem \ref{thm:localOp}.

\subsection{Rational points near patches of quadric hypersurfaces}
\label{ssec:quadrics}
Our first result is the following:
\begin{thm}
\label{thm:locOp_precise}
Choose $n \ge 3$ and an indefinite quadratic from $R$ on $\R^n$. Choose a basis $\zz, \ee, \ee_3, \dots, \ee_n$ such that
\[R(\zz + t \ee) = Bt^2\]
for some fixed nonzero constant $B$. Taking $A_1, \dots, A_{n-2}$ to be a basis for the linear functions on $\R^n$ that vanish on $\zz$ and $\ee$, we assume no nonzero polynomial of the form
\begin{equation}
\label{eq:pencil}
aR + \sum_{i \le j \le n-2} a_{ij} A_iA_j
\end{equation}
is rational. We also assume that there is some $\vv$ in $\R^n$ such that 
\begin{equation}
\label{eq:nonsingular_alg}
R(\zz + t\vv) = ct + dt^2\quad\text{for some nonzero } c.
\end{equation}
For $i$ satisfying $0 \le i \le n$, choose real numbers $a_i, b_i$ such that $a_i < b_i$. 

Then, given $\epsilon > 0$, there is some $Q_0 > 0$ such that, for $Q > Q_0$, if we take $\mathcal{Q}$ to be the set of points in $\R^n$ of the form
\begin{equation}
\label{eq:xx_defn}
 \xx =  q (\zz + c_2 \ee + c_3 \ee_3 + \dots + c_n \ee_n)
\end{equation}
such that
\begin{equation}
\label{eq:xx_req}
\left( R(\xx), \,\,qQ^{-1},\,\, c_2Q^{1/2},\, \,c_3Q,\,  \dots,\, \,c_n Q\right) \in  [a_0, b_0] \times \dots \times [a_n, b_n],
\end{equation}
we have
\[\# \left((\Z^n)_{\textup{prim}} \cap \mathcal{Q}\right) \,\ge \,  (1 - \epsilon) \zeta(n)^{-1} \cdot \vol\, \mathcal{Q}.\]
\end{thm}
We note that this theorem directly implies Theorem \ref{thm:localOp}.

For our work on hypersurfaces, we will also need a uniform variant of this theorem, which we will prove first.

\begin{thm}
\label{thm:locOp_uniform}
Fix $n \ge 3$ and a compact subset $K$ of $\R^n$ such that, for every nonzero integer quadratic form $P$ on $\R^n$, we have $P(\zz) \ne 0$ for all $\zz$ in $K$. Also fix some $C > 0$ and $\epsilon > 0$. Then there is some positive $Q_0$ so we have the following:

Choose an indefinite quadratic form $R$ on $\R^n$ and a basis $\zz, \ee, \ee_3, \dots, \ee_{n}$ of unit vectors for $\R^n$, with $\zz$ lying in $K$. We assume that the parallelepiped
\[\left\{c_1 \zz + c_2 \ee + c_3 \ee_3 + \dots + c_n \ee_n\,:\,\, |c_i| \le C\,\text{ for } \,1 \le i \le n\right\}\]
contains all unit vectors in $\R^n$.  We also assume that there is some constant $A$ satisfying $|A| \ge C^{-1}$ such that
\[R(\zz + t \ee) = At^2 \quad\text{for all } \, t \in \R.\]
Finally, we assume that, for some $i \ge 3$, we have
\[\left| R(\zz + \ee_i) - R(\zz) - R(\ee_i)\right| \ge C^{-1}.\]

For $0 \le i \le n$, choose real numbers $b_i \ge a_i$ of magnitude at most $C$. Choose $Q > Q_0$, and define $\mathcal{Q}$ to be the set of points of the form \eqref{eq:xx_defn} such that  \eqref{eq:xx_req} holds. Then
\[\# \left((\Z^n)_{\textup{prim}} \cap \mathcal{Q}\right) \,\ge \,  - \epsilon Q^{1/2} \,+\, \zeta(n)^{-1} \cdot \vol\, \mathcal{Q}.\]
\end{thm}
\begin{proof}
We fix $(n, K, C, \epsilon)$ as in the theorem statement throughout this proof. 

Given $R$, $\zz$, and $\ee$ as in the theorem statement, take $\langle\,\,,\,\, \rangle$ to be the quadratic form associated to $R$, so 
\[\langle \vv, \uu \rangle = R(\vv + \uu) - R(\vv) - R(\uu)\quad\text{for all}\quad\vv, \uu \in \R^n.\]
 We define an element  $w$ of $\sl_n (\R)$ associated to $(R, \zz, \ee)$ by
\[\zz \mapsto 0, \qquad \ee \mapsto  \langle \ee, \ee \rangle \zz,\qquad \ee_i \mapsto \langle \ee_i, \ee \rangle \zz - \langle \ee_i, \zz \rangle \ee \,\,\text{ for } i \ge 3.\]
With this definition, we may check that $\langle w \vv, \vv \rangle = 0$ for all $\vv \in \R^n$. This implies that the $1$-parameter unipotent subgroup $U = \{\exp(tw)\,:\, t \in \R\}$ preserves the quadratic form $R$.

As $(R, \zz, \ee)$ varies, we find that the associated elements $w$ lie in a compact subset of the set $\mathscr{W}_0$ defined in Notation \ref{notat:uniform_parabola}, as any nonzero quadratic form preserved by $U$ must be $0$ on $\zz$ and is hence irrational.  This will allow us to apply  Theorem \ref{thm:uniform_parabola} uniformly to all possible associated $w$.

Define the hyperplane $L_w$ of $\R^n$ as in Notation \ref{notat:uniform_parabola}. For $\vv$ in $L_w$, we take $I(Q, \vv)$ to be the set of real $t$ such that $\exp(tw)\vv$ lies in $\mathcal{Q}$ and $\mathcal{Q}(Q, \vv)$ to be the associated subset of $\mathcal{Q}$. There is some $C_0 > 0$ depending only on $(n, K, C, \epsilon)$ such that, if we take 
\[L_{0w} = \big\{ \vv \in L_w\,:\,\, \emptyset \ne I(Q, \vv) \subseteq\left [-C_0 Q^{1/2}, \,C_0 Q^{1/2}\right]\big\}, \]
and if we take $\mathcal{Q}_0$ to be the union of the $\mathcal{Q}(Q, \vv)$ over the $\vv$ in $L_{0w}$, we have
\[\vol\, \mathcal{Q} \backslash \mathcal{Q}_0 \le \tfrac{1}{2} \epsilon Q^{1/2}.\]
For $\vv$ in $L_{w}$, $I(Q, \vv)$ is the union of at most two closed intervals. Given $c > 0$, take $I_1(Q, \vv)$ to be the intersection of the $I(Q, \vv_1)$ taken over all $\vv_1 \in L_w$ within a distance of $c$ of $\vv$, and take $\mathcal{Q}_1(Q, \vv)$ to be the associated subset of $\mathcal{Q}(Q, \vv)$.  If $\mathcal{Q}_1$ is taken to be the union of the $\mathcal{Q}_1(Q, \vv)$ over $L_{0w}$, we may choose $c > 0$ sufficiently small given just $(n, K, C, \epsilon, C_0)$ so that
\[\vol\, \mathcal{Q}_0 \backslash \mathcal{Q}_1 \le \tfrac{1}{4} \epsilon Q^{1/2}.\]
Take $B_0$ to be  a closed solid hypercube in $L_w$ of sidelength $(n-1)^{-1/2}c$. We may decompose $L_w$ as a union of translates of $B_0$ subject to the restriction that any two distinct translates in the decomposition meet only at their boundary. Further, for each translated hypeprcube $B$, if we take $I$ to be the intersection of the $I(Q, \vv)$ over the $\vv$ in $B$, we find that $U_I B$ is contained in $\mathcal{Q}$ and contains the portion of $\mathcal{Q}_1$ over $B$. By applying Theorem \ref{thm:uniform_parabola}, we find for any $\epsilon_1 > 0$ that
\[\#(U_IB^{\circ} \cap (\Z^n)_{\text{prim}}) \ge -\epsilon_1 Q^{1/2} + \zeta(n)^{-1} \cdot \vol\, U_IB\]
so long as $Q$ is sufficiently large given $(n, K, C, \epsilon, \epsilon_1, c)$, where $B^{\circ}$ denotes the interior of $B$. The result follows by summing over the translates of $B$ containing some point in $L_{0w}$.
\end{proof}

\begin{proof}[Proof of Theorem \ref{thm:locOp_precise}]
We define $w$ as in the proof of Theorem \ref{thm:locOp_uniform}. 

Suppose $P$ is a nonzero quadratic form on $\R^n$ preserved by $w$. We claim that it takes the form \eqref{eq:pencil}.

Take $\langle\,\,,\,\, \rangle_P$ and $\langle\,\,,\,\, \rangle_R$ to be the quadratic forms associated to $P$ and $R$. By adjusting $P$ by some multiple of $R$, we may assume that $\langle\ee, \ee \rangle_P = 0$.

 We have $\langle \vv, w\vv \rangle_P = 0$ for all $\vv$, so $\langle \vv_1, w\vv_2 \rangle_P + \langle \vv_2, w\vv_1 \rangle_P = 0$ for all $\vv_1, \vv_2$. Applying the former to $\vv = \ee$ and the latter to $\vv_1 = \ee$ and $\vv_2 = \zz$ gives
\[ \langle \zz, \zz \rangle_P =  \langle \zz, \ee \rangle_P = 0.\]
Applying the latter now to arbitrary $\vv_1$ and $\vv_2 = \ee$ gives
\[\langle \vv, \zz \rangle_P = 0.\]
If $\vv$ is in the span of $\ee_3, \dots, \ee_n$, we then have
\[0 = \langle \vv,  w \ee_v \rangle_P  = -\langle \vv , \zz\rangle_R \cdot  \langle \vv, \ee\rangle_P.\]
We also note that $\langle \ee_i, \zz \rangle_R$ is nonzero for some $i$ by \eqref{eq:nonsingular_alg}. So $\langle \ee_i, \ee\rangle_P$ is zero for all $i \ge 3$. This implies that $P$ is expressible in the form \eqref{eq:pencil}.

With this checked, we find that $w$ preserves no nontrivial rational quadratic form, and the proof of Theorem \ref{thm:locOp_uniform} goes through. The only concern is that, unlike in that theorem, we have not assumed that $P(\zz)$ is nonzero for all nonzero rational quadratic forms $P$. But this was only used to be able to apply Theorem \ref{thm:uniform_parabola} uniformly, which we do not need to prove Theorem \ref{thm:locOp_precise}.
\end{proof}

\subsection{The proof of Theorem \ref{thm:main}}
\label{ssec:main}
We may assume without loss of generality that $\mathcal{M}$ is given in Monge form
\[\{ (\xx, F(\xx))\,:\,\, \xx \in B\},\]
where $B$ is an open solid hypercube in $\R^{n-1}$ and $F: \R^{n-1} \to \R$ is a $C^4$ function.

For any $c_0 > 0$, take $\mathcal{M}_0$ to be the subset of $(\xx, F(\xx))$ in $\mcM$ such that, for some unit vector $\ee$ in $\R^{n-1}$,  we have 
\begin{equation}
\label{eq:ee_curv_cond}
\left| \frac{\partial^2}{\partial^2 t}F(\xx + t \ee)\Big|_{t = 0}\,\right|  \ge c_0.
\end{equation}
Then, for $c_0$ sufficiently small, we have
\[\vol \,\mcM_0 \ge (1 - \delta/2)\vol\, \mcM \backslash \mcM_{\text{flat}}.\]
We then may decompose $\mcM_0$ into patches over hypercubes with negligible remainder. In this way, we find that we may assume without loss of generality that we have chosen $\ee$ such that \eqref{eq:ee_curv_cond} holds for all $\xx$ in $B$, so $\mcM_{\text{flat}}$ is empty.

Choose a closed ball $B_0$ whose interior contains $\mcM$. Take $\mcM_1, \mcM_2, \dots$ to be an enumeration of the intersections of the rational quadric hypersurfaces with $B_0$. We adopt the notation $\mcM_{i\epsilon}$ for the $\epsilon$-thickening of $\mcM_i$. Then we may choose positive numbers $\epsilon_1, \epsilon_2, \dots$ such that
\[\vol \, \mcM \backslash \left(\mcM_{1\epsilon_1} \cup \mcM_{2\epsilon_2} \cup \dots \right) \ge (1 - \delta/3) \vol\, \mcM \backslash \mcM_{\text{quad}}.\]
We furthermore may assume that the $\epsilon_i$ decrease quickly enough so that
\[2\vol\, \mcM_{i+ 1 \, \epsilon_{i+1}} \,\le\,  \vol\, \mcM_{i \epsilon_i} \quad\text{for } i \ge 1.\]
Then, if we choose $c' > 0$ sufficiently small and take
\[K = B_0 \backslash \left(\mcM_{1\, c'\epsilon_1} \cup \mcM_{2 \, c'\epsilon_2} \cup \dots\right)^{\circ},\]
we may assume
\begin{equation}
\label{eq:K_big}
\vol\, \mcM_{\epsilon} \cap K \ge (1 - \delta/2) \cdot 2\epsilon \cdot \vol \,\mcM \backslash \mcM_{\text{quad}}
\end{equation}
for all sufficiently small $\epsilon$.

To finish the proof, we will decompose $\mcM_{\epsilon}$ into patches. For each patch that meets $K$, we will show that the number of rational points near it is not too much smaller than what would be predicted heuristically. Theorem \ref{thm:main} then follows by summing over the patches by \eqref{eq:K_big}.

\begin{notat}
\label{notat:dQ_cube}
Fix a basis $\ee, \ee_2, \dots, \ee_{n-1}$ for $\R^{n-1}$ and positive numbers $\delta, Q$. Choose $\vv$ in $B$ so that
\[Y = \left\{ \vv + c_1 \ee + \dots + c_{n-1} \ee_{n-1}\,:\,\, (c_1, \dots, c_{n-1}) \in [0, \delta Q^{-1/2}] \times [0, \delta Q^{-1}]^{n-2}\right\}\]
is contained entirely in $B$. Choose a real number $c$. We then define
\[\mcM(\delta, Q, \vv, c) = \left\{(\xx,  y) \in \R^n\,:\,\, \xx \in Y \,\text{ and }\, 0 \le y - F(\xx) - c \le \delta Q^{-2}\right\}.\]
We call this a $(\delta, Q)$-\emph{parallel patch} to $\mcM$. We then take
\[\mcC(\delta, Q, \vv, c) = \big\{(q, \zz) \in \R^{n+1}\,:\,\, Q \le q \le (1 + \delta)Q \,\text{ and }\, \zz/q \in \mcM(\delta, Q, \vv, c)\big\}.\]
We call this a $(\delta, Q)$-\emph{parallel patch cone}.
\end{notat}

\begin{lem}
\label{lem:parallel_patch}
Take all notation and assumptions as above. Then there is $C > 0$ depending only on $\mcM$ and $K$ so we have the following:

Choose $\delta > 0$. Then there is $Q_0 > 0$ depending on $\mcM$, $K$, and $\delta$ such that, for all $Q > Q_0$ and any $(\delta, Q)$-parallel patch $\mcM(\delta, Q, \vv, c)$ that meets $K$, we have
\[\# \left(\mcC(\delta, Q, \vv, c) \cap (\Z^{n+1})_{\textup{prim}}\right) \ge (1 - C \delta)\cdot\zeta(n+1)^{-1} \cdot  \vol\,  \mcC(\delta, Q, \vv, c).\]
\end{lem}
\begin{proof}
We want to choose a quadratic form $R$ on $\R^{n+1}$ that is small on the parallel patch cone $\mcC(\delta, Q, \vv, c)$. To start, choose 
\[(\xx,  x_n) = (x_1, \dots, x_{n})\quad\text{in}\quad\mcM(\delta, Q, \vv, c) \cap K,\]
 and define $G: \R^{n-1} \to \R$ by
\[G(u_1, \dots, u_{n-1}) = F\left(\xx + u_1 \ee + u_2 \ee_2 + \dots + u_{n-1} \ee_{n-1}\right).\]

We want to choose $R$ so that
\begin{align}
\label{eq:taylor} &R\big(1,\, \xx + u_1 \ee_1 + \dots + u_{n-1} \ee_n,\, G(u_1, \dots, u_{n-1}) + x_n - F(\xx) \big)\\
& = \mathcal{O}\left(u_1^4 + u_2^2 + \dots + u_{n-1}^2\right)\nonumber
\end{align}
for small values of the $u_i$, where the implicit constant depends on $\mcM$ but not on the choice of parallel patch. Calling this function $R^{\circ}(u_1, \dots, u_{n-1})$, we see that this amounts to checking
\begin{align}
&R^{\circ}(0) = R^{\circ}_{111}(0) = 0 \quad\text{and} \label{eq:derivative_conds}\quad R^{\circ}_{i}(0) = R^{\circ}_{1i}(0) =  0 \, \text{ for }\, i \le n-1.
\end{align}
Here, we use multi-index notation for the partial derivatives, so $R^{\circ}_{1i}$ is notation for $\frac{\partial^2}{\partial u_1 \partial u_i} R^{\circ}$, etc.

This is a collection of $2n$ linear conditions, while the space of all quadratic forms on $n+1$ variables is $\tfrac{1}{2}(n^2 + n)$ dimensional, so we know that some $R$ can be found. More specifically, if we define $R$ by
\begin{align*}
&R\left(u_0, \,v_0\xx +  v_1\ee + v_2 \ee_2 + \dots + v_{n-1}\ee_{n-1}, \,x_n + v_n\right) \,=\, \sum_{i = 1}^n a_{0i} v_0v_i + \sum_{i =1}^n a_{1i}v_1v_i
\end{align*}
with
\begin{align*}
&a_{0n} =-3G_{11}(0), \quad a_{1n} = G_{111}(0),\\
&a_{11} = \tfrac{3}{2} G_{11}(0)^2  - G_{1}(0) G_{111}(0),\\
&a_{1i} =  3 G_{1i}(0) G_{11}(0) - G_i(0) G_{111}(0)\quad\text{for } \,2 \le i \le n-1, \quad\text{and}\\
&a_{0i} = 3G_{11}(0) G_{i}(0) \quad\text{for }\, i \le n-1,
\end{align*}
we find that $R$ satisfies the conditions \eqref{eq:derivative_conds}.

Take
\[a =  F(\xx) + c - \xx_n \quad\text{and}\quad b = \delta Q^{-2} + F(\xx) + c - \xx_n.\]
These have magnitude at most $\delta Q^{-2}$. For convenience, assume $G_{11}(0)$ is negative. From \eqref{eq:taylor}, we find that there are positive numbers $C_0, C_1$ such that
\[\mcM_{\approx}:= \left\{ (\zz, z_n) \in Y \times \R \,:\,\, a_{0n}a + C_0 \delta^2 Q^{-2}   \le R(1, \zz, z_n) \le   a_{0n} b + C_0 \delta^2 Q^{-2}\right\}\]
is contained in $\mcM(\delta, Q, \vv, c)$, and that
\[\vol\, \mcM_{\approx}  \ge (1 - C_1 \delta)\vol\, \mcM(\delta, Q, \vv, c).\]
Here, neither $C_0$ nor $C_1$ depends on $\delta$, $Q$, or the choice of parallel patch.

We note that the parallel patch cone has volume at least $c_1 \delta^{n+1} Q^{1/2}$, where $c_1 > 0$ does not depend on $\delta$, $Q$, or the choice of parallel patch. Then the lemma follows from applying Theorem \ref{thm:locOp_uniform} uniformly to the possible choices of $\mcM_{\approx}$.
\end{proof}

We now may prove Theorem \ref{thm:main}, which fixes a choice of $\delta > 0$ and $\mcM$. As above, we may assume that \eqref{eq:ee_curv_cond} is satisfied, and we choose $K$ as above so \eqref{eq:K_big} is satisfied for sufficiently small $\epsilon$.

Choose $\delta_1 > 0$. Then, for $Q > 1$ and $\epsilon$ in $[\delta Q^{-2},\, Q^{-1}]$, we may decompose the cone $\mcC_Q\left(\mcM_{\epsilon}\right)$ as a disjoint union of $(\delta_1, Q_1)$-parallel patch cones with $Q_1$ no smaller than $\delta_1 Q$, together with a remainder of volume at most $C_0 \delta_1 \vol\, \mcC_Q\left(\mcM_{\epsilon}\right)$, where $C_0 > 0$ depends on $\mcM$ and $\delta$ but not on $\delta_1$, $Q$, or $\epsilon$.

By \eqref{eq:K_big}, the union of these parallel patch cones not above a parallel patch meeting $K$ has volume at most $\tfrac{2}{3}\delta \vol\, \mcC_Q\left(\mcM_{\epsilon}\right)$ so long as $Q$ is sufficiently large relative to $\mcM$. Applying Lemma \ref{lem:parallel_patch} to the remainder, we find there is $C_1 > 0$  depending only on $\mcM$ so that, for all $Q$ larger than some bound determined from $\mcM$  and $\delta_1$,
\[\#\left(  \mcC_Q\left(\mcM_{\epsilon}\right) \cap \left(\Z^{n+1}\right)_{\text{prim}} \right) \ge \zeta(n+1)^{-1} \cdot \left(1 - \tfrac{2}{3}\delta - (C_0 + C_1)\delta_1\right) \vol\, \mcC_Q\left(\mcM_{\epsilon}\right).\]
We then may take $\delta_1$ small enough that the conclusion of the theorem holds for all sufficiently large $Q$. \qed

\subsection{The proof of Theorems \ref{thm:hypersurface_qual}, \ref{thm:main_eff}, and \ref{thm:planarcurves}}
\label{ssec:RAG}
With the central result Theorem \ref{thm:main} shown, we now prove the other main theorems of this paper. We start with Theorem \ref{thm:hypersurface_qual}, where some work is required to handle hypersurfaces contained in unions of rational quadrics.

\begin{proof}[Proof of Theorem \ref{thm:hypersurface_qual}]
By Theorem \ref{thm:main}, we may focus on the case that $\mcM$ is contained in $\mcM_{\text{flat}} \cup \mcM_{\text{quad}}$. In the case that $\mcM$ is contained in $\mcM_{\text{flat}}$, it is straightforward to show that $\mcM$ is contained in a hyperplane since it is a connected manifold. So we may assume the curvature is nonzero at some point in $\mcM$.

Take $P_1, P_2, \dots$ to be an enumeration of the integral quadratic polynomials up to scalar multiple, and take $\mcM_i$ to be the zero locus of $P_i$ in $\mcM$. This is a closed set. Since $\mcM$ is $C^4$, the interiors $\mcM_i^{\circ}$ and $\mcM_j^{\circ}$ have disjoint closures for all $i \ne j$. So, for any point $x$ in $\mcM$ with nonzero curvature tensor, there is either a unique $i$ so that $x$ is contained in $\mcM_i^{\circ}$, or there is no finite collection of $\mcM_i$ whose union contains a neighborhood of $x$.

Assuming that no $P_i$ vanishes on all of $\mcM$, we may find an $x$ where $\mcM$ has nonzero curvature such that the latter condition holds. Choose a set of points $x_1, x_2, \dots$ converging to $x$ such that $x_i$ is not in $\mcM_j$ for any $j < i$. According to the procedure in Section \ref{ssec:quadrics}, for sufficiently large $i$, we may associate each $x_i$ with an element $w_i$ in $\sl_{n+1}(\R)$. Take $\mu_i$ to be the measure on $\SL_{n+1}(\R)/\SL_{n+1}(\Z)$ defined so 
\[\frac{1}{T}\lim_{T \to \infty} \int_0^{\infty} f(u_t) dt = \int f d\mu_i\]
 for all $f$ in $C_c(\SL_{n+1}(\R)/\SL_{n+1}(\Z))$. 

Take $\mu$ to be a weak${}^*$ limit of some infinite sequence of $\mu_i$. By \cite[Theorem 1.1]{MoSh95}, $\mu$ is invariant under the left action of some Lie group $H$ that contains $\exp(w_i)$ for infinitely many $i$. Take $\mfh$ to be the Lie algebra associated to $H \subseteq \SL_{n+1} (\R)$.

Take $W$ to be the set of nilpotent $w$ in $\mfh$ of rank at most $2$. Then $w$ is a real algebraic set. The subset of $w$ in $W$ that preserve a given nonzero integral quadratic form is also a real algebraic set, as is the subset of $w$ such that $w^2 = 0$. 

From the construction of $H$, we know that $W$ is not contained in any finite union of these algebraic subsets. So $W$ must contain an irreducible component $W_0$ such that the intersection of any one of these subsets with $W_0$ has positive codimension \cite[Theorem 2.8.3]{Boch13}. From a measure-theoretic argument, $W_0$ is not contained in the (countable) union of all of these subsets.

So there is some $w$ in $W$ such that $w^2$ is nonzero and such that $w$ preserves no nonzero integral quadratic form. Then Theorem \ref{thm:uniform_parabola_func} shows that
\[\lim_{T \to \infty} \frac{1}{T} \int_0^T \sum_{v \in (\Z^{n+1})_{\text{prim}}}  f\left(\exp(tw) g v\right) \ge  \zeta(n+1)^{-1} \int_{\R^n} f(x)dx\]
for $f$ nonnegative and continuous with compact support and any $g$ in $\SL_n(\R)$. Then, for any such $f$ besides $0$
\[\limsup_{i \to \infty} \int  \sum_{v \in (\Z^{n+1})_{\text{prim}}}  f\left(gv\right)  d\mu_i(g)  \ge \int \sum_{v \in (\Z^{n+1})_{\text{prim}}}  f\left(gv\right)  d\mu(g) > 0.\]
This is enough to show that, for large enough $i$, the thickened parabolas corresponding to the local approximation to the surface at $x_i$ contain primitive integer points.
\end{proof}

The proof of Theorem \ref{thm:main_eff}, an effective rational point count for curves, largely follows the proof of Theorem \ref{thm:main} given above. The extra tool we need is the following, which uses the notation from Section \ref{ssec:eff}:
\begin{lem}
\label{lem:Qhugger}
There is an absolute $C > 0$ so we have the following:

Take $\mathcal{M}$ to be a $C^5$ curve in $\R^2$, and take $\mathcal{M}_0$ as in Theorem \ref{thm:main_eff}. Then, for any $\delta > 0$, there is $T_0 > 10$ so that the subset of $x \in \mathcal{M}_0$ such that the corresponding element $w_x \in \sl_3(\R)$ lies in $\mathscr{W}(6, T_0)$ has arclength at least $(1 - \delta) \cdot \textup{arclength}(\mcM_0)$.
\end{lem}
\begin{proof}
For $x$ in $\mathcal{M}_0$, the osculating conic has order of contact four with $\mathcal{M}$. Taking $\phi: \R^2 \to \R^5$ to be the map
\[(x, y)\mapsto (x, y, x^2, xy, y^2),\]
this implies that $\phi(\mathcal{M}_0)$ is a nondegenerate $C^5$ curve at $\phi(x)$.

By \cite[Theorem A]{Klein98}, we find that, in some neighborhood $U \subseteq \mathcal{M}$ of $x$, the set of $y \in U$ such that $w_y$ lies outside $\bigcap_{T_0 > 10}\msW(6, T_0)$ is negligible. Then, for sufficiently large $T_0$, we find that $\msW(6, T_0)$ contains a subset of $\mathcal{M}_0$ of volume $(1 - \delta) \vol\, \mathcal{M}$.
\end{proof}

With this proved, Theorem \ref{thm:main_eff} now follows from Theorem \ref{thm:R3_thickened} as in the argument in Section \ref{ssec:main}. Theorem \ref{thm:planarcurves} then follows immediately. After all, if the set $\mcM_0$ defined in Theorem \ref{thm:main_eff} is negligible, it then follows that every point on $\mcM$ where the curvature is nonzero  has some neighborhood in $\mcM$ contained in a conic. This then implies that $\mcM$ is covered by a finite collection of lines and conics. Since it is also connected and $C^5$, we find that some single quadratic polynomial vanishes on $\mcM$, establishing the theorem by contradiction. \qed

\begin{rmk}
The $C^4$ condition in Theorem \ref{thm:main} was only used in the proof of Lemma \ref{lem:parallel_patch}. Indeed, this lemma, and hence the theorem, remains true even if $\mcM$ is assumed only to be $C^3$, so long as the involved $3^{\text{rd}}$ derivatives are assumed to be Lipschitz continuous. It is unclear to the author whether there should exist $C^3$ counterexamples to this theorem.

On the other hand, it is straightforward to construct $C^3$ counterexamples to Theorem \ref{thm:hypersurface_qual} by stitching together patches of rational quadrics.

Finally, the extra $C^5$ condition in Theorems \ref{thm:main_eff} and \ref{thm:planarcurves} is only used in the proof of Lemma \ref{lem:Qhugger} to eliminate curves that spend too long near rational conics. Our guess would be that some extra condition is needed beyond $C^4$ to eliminate such pathological curves, but we have not shown this.
\end{rmk}

\section{Semisimple Lie algebras}
\label{sec:semisimple}
As a first step towards proving Theorem \ref{thm:unRatnered}, we will prove a classification result for the semisimple part of the associated Lie algebras.

\begin{thm}
\label{thm:semisimple}
Take $V$ to be a finite dimensional vector space over $\QQ$. Choose a nonzero nilpotent element $w$ in $\slV \otimes \R$. We will assume that $w$ has rank at most $2$ as an endomorphism of $V \otimes \R$, and we assume that $w^2$ is nonzero if $w$ has rank $2$.

Take $\mfh \subseteq \sl(V)$ to be the minimal Lie subalgebra of $\slV$ such that $\mfh \otimes \R$ contains $w$. We will suppose that $\mfh$ is semisimple and that the kernel of $\mfh$ is trivial.

Then there is some number field $K$ with a real embedding such that $(\mfh, V)$ is identifiable with 
\[\left(\mathfrak{sl}( K^{n/d}), K^{n/d}\right)\quad\text{or}\quad \left(\mathfrak{so}_Q\,( K^{n/d}),\,K^{n/d}\right),\]
where $d$ denotes the degree of $K$ over $\QQ$ and where $Q$ denotes some nondegenerate quadratic form on $K^{n/d}$. In the orthogonal case, $w$ must have rank $2$.
\end{thm}
To prove this, we first prove the analogous result over $\C$.

\begin{thm}
\label{thm:ss_complex}
Take $V$ to be a finite dimensional irreducible faithful representation of a complex semisimple Lie algebra $\mfh$. We suppose $\mfh$ contains a nonzero nilpotent element $w$ of rank at most $2$. If $w$ has rank $2$, we suppose $w^2$ is nonzero.

Then either $\mfh$ equals $\slV$, or $w$ has rank $2$  and there is some nondegenerate quadratic form $Q$ on $V$ such that $\mfh = \mathfrak{so}_Q(V)$.
\end{thm}

Fans of Dynkin diagrams may notice that this theorem claims that $\mfh$ needs to be simple except in one special case. We handle this case first.
\begin{prop}
Given $w$, $\mfh$ and $V$ satisfying the conditions of Theorem \ref{thm:ss_complex}, if $\mfh$ is not simple, then $V$ is $4$-dimensional, $w$ has rank $2$, and $\mfh = \mathfrak{so}_Q(V)$ for some nondegenerate quadratic form $Q$ on $V$.
\end{prop}

\begin{proof}
Write $\mfh$ as a direct sum $\mfh_1 + \mfh_2$ of nontrivial semisimple Lie algebras. From the parameterization of the irreducible representations of semisimple Lie algebras in terms of weights, we know $V$ is a subrepresentation of $V_1 \otimes V_2$, where $V_i$ is an irreducible representation of $\mfh_i$. By the Weyl dimension formula \cite[24.3]{Hump72}, we find that $V = V_1 \otimes V_2$. As nontrivial representations of semisimple Lie algebras, $V_1$ and $V_2$ are at least $2$-dimensional.

Write $w$ in the form $(w_1, w_2)$ in $\mfh_1 + \mfh_2$. Then, for $v_1 \in V_1$ and $v_2 \in V_2$,
\[w(v_1 \otimes v_2) = w_1(v_1) \otimes v_2 + v_1 \otimes w_2(v_2).\]
If $w_2 = 0$, then the image of $w$ contains $w_1(V_1) \otimes w_2$. This image has dimension at least $2$; if $w^2$ is nonzero, it must have dimension at least $4$. Neither is possible, so $w_2$ is nonzero. Similarly, $w_1$ is nonzero.

Fix $v_2 \in V_2 $ so $w_2(v_2)$ is not a multiple of $v_2$, as is possible since $w_2$ is nonzero and nilpotent. Then $v_1 \mapsto w(v_1 \otimes v_2)$ is an injective map from $V_1$ to the image of $w$. So $w$ has rank $2$, and $V_1$ has dimension $2$. Similarly, $V_2$ has dimension $2$.

The only irreducible two-dimensional representation of a simple complex Lie algebra is the standard representation of $\sl_2(\C)$, so $\mfh_1 \cong \mfh_2 \cong \mathfrak{sl}_2 (\C)$ and $V_i$ is given by the standard representation of $\mfh_i$. Then $\mfh$ may be identified with the orthogonal Lie algebra associated to some nondegenerate quadratic form on the four-dimensional space $V$ \cite[18.2]{FH04}.
\end{proof}

\subsection{The case of simple  $\mfh$}
We now wish to prove Theorem \ref{thm:ss_complex} under the condition that $\mfh$ is simple. Choose a Borel subalgebra $\mfb$ of $\mfh$ containing $w$ \cite[16.3]{Hump72}. We may write $\mfb$ in the form
\[\mfb = \mft \oplus \bigoplus_{\alpha \in \Phi^+} L_{\alpha},\]
where $\mft$ is a maximal toroidal subalgebra of $\mfh$, where $\Phi^+$ denotes the set of positive roots in a root system
\[\Phi \subseteq \mft^* := \Hom\left(\mft, \C\right)\]
associated to $\mft$, and where $L_{\alpha}$ denotes the $\alpha$-eigenspace of $\mfh$ for any $\alpha$ in $\Phi$.

Given an irreducible finite dimensional $\mfh$-representation $W$ and $\beta \in \mft^*$, we take $W_{\beta}$ to be the the $\beta$-eigenspace of $W$. The set of $\beta$ such that $W_{\beta}$ is nonzero for some such $W$ forms a lattice $\Lambda$ in the rational span of the roots. The roots generate a sublattice $\Lambda_r$ of this lattice.

We then have a decomposition
\[V = \bigoplus_{\beta \in \Lambda} V_{\beta}.\]
Among the $\beta$ such that $V_{\beta}$ is nonzero, there is a unique choice of $\beta$ such that $V_{\beta + \alpha} = 0$ for all $\alpha \in \Phi^+$. Furthermore, this $\beta$ determines the representation, and $V_{\beta}$ is one-dimensional \cite[20.1]{Hump72}.  We call it the \emph{greatest weight} of $V$. This lies in the Weyl chamber of $\Lambda$.

Take $\Lambda(V)$ to be the collection of $\beta$ in $\Lambda$ such that $V_{\beta}$ is nonzero. This is the least set containing the greatest weight of $V$ such that the following properties are satisfied:
\begin{enumerate}
\item Given $\tau$ in the Weyl group associated to $\Phi$, if $\beta$ is in $\Lambda(V)$, so is $\tau \beta$.
\item If $\beta$ lies in $\Lambda(V)$ and $\alpha$ lies in $\Phi$, and if $\beta + \alpha$ is in the convex hull of $\Lambda(V)$, then $\beta + \alpha$ lies in $\Lambda(V)$.
\end{enumerate}
See \cite[21.2 and 21.3]{Hump72} for more details.

We write 
\[w = \sum_{\alpha \in \Phi^+} w_{\alpha},\]
 where $w_{\alpha}$ lies in $L_{\alpha}$ for each $\alpha$ in $\Phi^+$. Given  $\beta$ in $\Lambda(V)$ and $\alpha$ in  $\Phi^+$, we have the key properties
\begin{alignat}{2}
\label{eq:translation}
&w_{\alpha} V_{\beta} \subseteq V_{\alpha + \beta}\quad&&\text{and}\\
\label{eq:nonvanishing}
&w_{\alpha}V_{\beta} \ne 0 \quad&&\text{if }\,\, w_{\alpha} \ne 0\,\,\text{ and }\, \,\alpha + \beta \in \Lambda(V).
\end{alignat}
See \cite[20.1 and 21.3]{Hump72}.

Take $\Lambda(w)$ to be the set of  $\alpha \in \Phi^+$ such that $w_{\alpha}V$ is nonzero. We call $\alpha \in \Lambda(w)$ \emph{minimal} if there is some total ordering on $\Lambda$ respecting addition such that every root in $\Phi^+$ is greater than $0$ and under which $\alpha$ is the minimal element of $\Lambda(w)$.
\begin{lem}
\label{lem:rank_bnd}
Choose a minimal root $\alpha$ in $\Lambda(w)$. Then the rank of $w$ is no smaller than the number of $\beta$ in $\Lambda(V)$ such that $\beta + \alpha$ is also in $\Lambda(V)$.
\end{lem}
\begin{proof}
Fix a total ordering as above such that $\alpha$ is minimal in $\Lambda(V)$. Take $\beta_1, \dots, \beta_r$ to be the distinct elements in $\Lambda(V)$ such that $\beta_1 + \alpha,  \dots, \beta_r + \alpha$ are also in $\Lambda(V)$. We assume $\beta_1 > \dots > \beta_r$ under our choice of ordering. Then we have
\begin{alignat*}{2}
&w V_{\beta_i} \quad\text{has nontrivial image in}\quad &&V_{\beta_i + \alpha} \quad\text{for } \,i \le r \quad\text{and}\\
&wV_{\beta_i} \quad\text{has trivial image in} &&V_{\beta_j + \alpha} \quad\text{for } \,i < j \le r.
\end{alignat*}
Then $w$ must have rank at least $r$.
\end{proof}

\begin{lem}
\label{lem:exceptional}
In the context of Theorem \ref{thm:ss_complex}, $\mfh$ cannot be an exceptional Lie algebra.
\end{lem}
\begin{proof}
Assume otherwise, so that $\mfh$ is exceptional. Since $w$ has positive rank, $\Lambda(w)$ is nonempty. Choose  a minimal root $\alpha$ in $\Lambda(w)$. In the case that $\Phi$ is $G_2$, we assume $\alpha$ is short if there is a short root in $\Lambda(w)$.

Choose $\beta$ in $\Lambda(V)$ so $\beta + \alpha$ is also in $\Lambda(V)$. Applying an element in the Weyl group if necessary, we may assume that $\beta$ is not a multiple of $\alpha$. 

Then, applying the theory of extended Dynkin diagrams \cite[Table 8]{Dynk57a} and the fact that the Weyl group is transitive on the roots of a given length, we may find a root subsystem $\Psi$ orthogonal to $\alpha$ such that
\[\Psi  \cong \left\lbrace\begin{array}{ll}
 A_1 \,\text{ if } \Phi \cong G_2& A_5 \,\text{ if } \Phi \cong E_6 \\ 
 C_3 \,\text{ if } \Phi \cong F_4 \text{ and } \alpha \text{ is a long root} \quad\qquad& D_6 \,\text{ if } \Phi \cong E_7 \\
 A_3 \,\text{ if } \Phi \cong F_4 \text{ and } \alpha \text{ is a short root}&A_7 \,\text{ if } \Phi \cong E_8.\\
\end{array}
\right.\]
Choosing $c \in \tfrac{1}{2}\Z$ so $\beta + c \alpha$ is orthogonal to $\alpha$, we see from Lemma \ref{lem:rank_bnd} that the rank of $w$ is at least equal to the number of elements in the orbit of $\beta + c\alpha$ under the Weyl group of $\Psi$. Given a root system of type $A_n$, we see that the orbit of any nonzero vector in the rational span of the roots under the Weyl group has size at least $n+1$; for $C_3$, such an orbit has size at least $6$; for $D_6$, it has size at least $12$. This handles all the possible exceptional cases except for $\Phi \cong G_2$.

In this case, there are distinct elements $\beta$ and $\beta'$ in $\Lambda(V)$ whose $\alpha$ components are the same and negative and such that $\beta + \alpha$ and $\beta' + \alpha$ are also in $\Lambda(V)$. If $\beta$ does not have $\alpha$ component $-1/2$, then $\beta + 2\alpha$ and $\beta' + 2\alpha$ are also in $\Lambda(V)$, so $w$ has rank at least $4$, contradicting the assumptions of Theorem \ref{thm:ss_complex}.

So we assume that $\beta$ has $\alpha$ component $-1/2$. If $\alpha$ was in $\Lambda(V)$, we would again find $w$ had rank at least $4$ by applying Lemma \ref{lem:rank_bnd} to $\{\beta, \beta', -\alpha, 0 \}$. So we are left with the case that $\beta$ has $\tau \alpha$ component $-1/2$ or $1/2$ for all $\tau$ in the Weyl group.

This forces $\alpha$ to be a long root and $\beta$ to be a short root. Since $w$ must have rank $2$,  $w^2$ must be nonzero. This implies that $\Lambda(w)$ must contain a second long root. Indeed, there is a unique long positive root that is the sum of other long positive roots; for $w^2$ to be nonzero, we find $\Lambda(w)$ must contain both of the other long positive roots. But we then find that $w$ has rank at least $4$, a contradiction.
\end{proof}

There is one special case that cannot be ruled out from the geometry of $\Lambda(V)$. We handle it next.
\begin{lem}
\label{lem:standard_symplectic}
In the context of Theorem \ref{thm:ss_complex}, if $\mfh$ is a symplectic Lie algebra and $w$ has rank $2$, then $V$ cannot be the standard representation of $\mfh$.
\end{lem}
\begin{proof}
Suppose otherwise. Then there is a nondegenerate skew-symmetric form $B$ on $V$ such that
\[B(wv_1, v_2) + B(v_1, wv_2) = 0 \quad\text{for all } \, v_1, v_2 \in V.\]
So
\[B(w^2v, v) = -B(wv, wv) = 0\quad\text{for all }\, v \in V.\]
The set of vectors $v$ such that $w^2v$ is nonzero generates $V$ since $w^2$ is nonzero, so the one dimensional space $w^2V$ must be in the kernel of $B$. But this contradicts the nondegeneracy of $B$.
\end{proof}

We now handle the classical cases of Theorem \ref{thm:ss_complex}.
\subsubsection*{Proof of Theorem \ref{thm:ss_complex}}
Take all the setup as above. By Lemma \ref{lem:exceptional}, we may assume that $\mfh$ is a simple classical Lie algebra. We may also assume that $\Lambda(w)$ is nonempty. Choose some minimal root $\alpha$ in $\Lambda(w)$. 

We give some standard setup for the classical root systems; the results we quote can be found in \cite{FH04}. If $\Phi$ is $B_n$, $C_n$, or $D_n$, we take $e_1, \dots, e_n$ to be the standard orthonormal basis to $\R^n$. The positive roots in $\Phi$ may then be identified with
\[ e_i - e_j \,\text{ and }\, e_i + e_j \,\text{ for }\, 1 \le i < j \le n \,\text{ together with } \begin{cases} e_i \,\,\text{ for }  i \le n&\text{ if } \Phi \cong B_n \\ 2e_i\,\, \text{ for } i \le n &\text{ if } \Phi \cong C_n \\ \text{Nothing else} &\text{ if } \Phi \cong D_n.\end{cases}\] 
In the case $A_n$, we instead take $e_1, \dots, e_{n+1}$ to be the standard orthonormal basis for $\R^{n+1}$ and identify the positive roots with
\[e_i - e_j \text{ for } \, 1 \le i < j \le n.\]
We take $e = \tfrac{1}{n+1}(e_1 + \dots + e_{n+1})$ in this case.

Take $\beta$ to be the greatest weight of $V$. We first prove the theorem in the cases
\begin{align*}
(\Phi, \beta) \,\,\,  \cong \,\,\, (A_n, \,e_1 - e),\,\, (B_n,\, e_1),\,\, (C_n,\,e_1),\,\, (D_n, e_1).
\end{align*}
These all correspond to the standard representation of a classical Lie algebra \cite{FH04}. In the $A_n$ case, this forces $\mfh \cong \slV$. In the $B_n$ and $D_n$, this forces $\mfh \cong \mathfrak{so}_Q V$ for some $Q$; we note in these cases that $w$ must have rank $2$ by Lemma \ref{lem:rank_bnd}. We have already handled the $C_n$ case in Lemma \ref{lem:standard_symplectic}.

Via Dynkin diagram isomorphisms, the above work for the standard representations allows us to conclude the theorem for
\begin{align*}(\Phi, \beta) \,\,  \cong \,\, &(A_1, \,2e_1 - 2e),\,\, (A_3, e_1 + e_2 - 2e),\,\, (A_n,\, e_1 + \dots + e_n-ne) ,\\
&\left(B_2,\, \tfrac{1}{2}(e_1+e_2)\right),\,\,\left(D_4,\, \tfrac{1}{2}(e_1 + e_2 + e_3 + e_4)\right),\,\, \left(D_4,\, \tfrac{1}{2}(e_1 + e_2 + e_3 - e_4)\right).
\end{align*}

All remaining cases may be handled by appealing to Lemma \ref{lem:rank_bnd}. Specifically, take $\beta_1, \dots, \beta_k$ to be weights in $\Lambda(V)$ with disjoint orbits under the Weyl group, and take $r(\beta_i, \alpha)$ to be the number of elements in the orbit of $\beta_i$ not orthogonal to $\alpha$. Then Lemma \ref{lem:rank_bnd} gives that $H$ has rank at least $\tfrac{1}{2}r(\beta_1, \alpha) + \dots + \tfrac{1}{2}r(\beta_k, \alpha)$. We note for convenience that $r(\beta_i, \alpha)$ only depends on $\alpha$ insofar as it depends on the length of $\alpha$.

\textbf{Case: $\Phi \cong A_n\quad$} 
In this case, we may write $\beta$ in the form
\[a_1(e_1 - e) + \dots + a_n(e_n - e)\]
with $a_1 \ge \dots \ge a_n$ nonnegative integers. Taking $a_{n+1} = 0$, the Weyl orbit of this weight consists of elements of the form
\[a_{\sigma(1)}(e_1 - e) + \dots + a_{\sigma(n)}(e_n - e),\]
where $\sigma$ is a permutation of $\{1, \dots, n+1\}$. 

Take $i \le n$ maximal so $a_i$ is nonzero. If $a_1 \ne a_i$, we see that $r(\beta, e_1 - e_2)$ is at least $6$, as there are at least $6$ possible values for the tuple $(a_{\sigma(1)}, a_{\sigma(2)})$ with $a_{\sigma(1)} \ne a_{\sigma(2)}$. So we may assume $a_1 = a_i$.

If $n = 1$, we find that $w$ has rank at least $a_1$, so our work above handles this case.

Assuming $a_1 = a_i$ and $n > 1$, we have
\[r(\beta, e_1 - e_2) = 2\binom{n-1}{n - i}\]
(first choose which of $a_{\sigma(1)}$ and $a_{\sigma(2)}$ is zero, then choose which of $a_{\sigma(3)}, \dots, a_{\sigma(n+1)}$ are zero). This is at least $6$ unless $i = 1$ or $i = n$, where it equals $2$,  or $i = 2$ and $n = 3$, where it equals $4$.

Suppose we are in one of these final cases with $a_1 = a_i$ and $i > 1$. We already handled these cases when $a_1 = 1$. So suppose $a_1 > 1$. If $i < n$, then $\beta_0  = \beta - e_i + e_{i+1}$ is in $\Lambda(V)$, and we find $r(\beta_0, \alpha) \ge 6$. The case $i = n$ is equivalent to that of $i = 1$ under the Dynkin diagram automorphism, finishing the proof for $A_n$.

\textbf{Case: $\Phi \cong B_n$ with $n \ge 2$\quad}
 We may write $\beta$ in the form
\[\tfrac{1}{2}(a_1e_1 + \dots + a_n e_n)\]
where the $a_1 \ge \dots \ge a_n$ are nonnegative integers of the same parity. The orbit of this under the Weyl group consists of the weights
\[\tfrac{1}{2}\left(\pm a_{\sigma(1)}e_1 \pm \dots \pm a_{\sigma(n)}e_n\right),\]
where the signs vary over all $2^n$ possibilities and $\sigma$ is a permutation of $\{1, \dots, n\}$. Take $i \le n$ maximal so $a_i$ is positive.  By considering the set of $a_{\sigma(i)}$ that equal zero, and by keeping track of the sign on the remaining coefficients, we have
\[r(\beta, e_1) \ge \binom{n-1}{n - i} 2^i\quad\text{and}\quad r(\beta, e_1 - e_2) \ge \binom{n-2}{n-i}2^{i-1} + \binom{n-2}{n - i - 1}2^{i+1}.\]
More specifically, the latter relationship is found by separately considering the case that $\tau \beta$ has nontrivial $e_1$ and $e_2$ component, and where it has one of these components trivial.

Both these expressions are at least $8$ for $n \ge 4$ unless $i = 1$. So suppose $i = 1$ and $n \ge 2$. We have handled the case $a_1 = 2$, $a_2 = 0$, and the case $a_1 > 2$, $a_2 = 0$ is straightforward since $\Lambda(V)$ also contains $\beta_2 = e_1 + e_2$ and $r(\beta, \alpha) + r(\beta_2, \alpha) \ge 6$.

For $n = 3$ and $i > 1$, both expressions are at least $8$ unless $i = 3$, where the former is $8$ but the latter is $4$. So, in this case, we need to handle the situation where $\Lambda(w)$ contains no short root. The case $a_1  > 1$ is straightforward, leaving the weight for the spin representation $\tfrac{1}{2}(e_1 + e_2 + e_3)$. Because $\Lambda(w)$ contains no short root, we find that $w$ fixes the subspaces
\[\bigoplus_{\substack{s_1, s_2, s_3 \in \pm 1 \\ s_1 \cdot s_2 \cdot s_3 = 1}} V_{\frac{1}{2}(s_1e_1 + s_2e_2 + s_3e_3)}  \quad\text{and}\quad \bigoplus_{\substack{s_1, s_2, s_3 \in \pm 1 \\ s_1 \cdot s_2 \cdot s_3 = -1}} V_{\frac{1}{2}(s_1e_1 + s_2e_2 + s_3e_3)}\]
and has nontrivial image in each of them, so $w^2$ is nonzero. But for this to happen, $w$ must have rank $2$ restricted to one of these subspaces, which would imply that $w$ has rank at least $3$.

This just leaves the case $n = 2$ with $a_2 > 0$. In this case, $r(\beta, \alpha)$ is at least $8$ unless $a_1 = a_2$, where it equals $2$ or $4$. We handled the case $a_1 = a_2 = 1$ above. Finally, if $a_1 = a_2 > 1$, then  $\Lambda(V)$ contains $\beta_0 = \beta - e_2$, and $r(\beta_0, \alpha) \ge 4$.

\textbf{Case: $\Phi \cong C_n$ with $n \ge 3$\quad}
 We may write $\beta$ in the form
\[a_1e_1 + \dots + a_ne_n,\]
where $a_1 \ge \dots \ge a_n$ are nonnegative integers. The orbit of this under the Weyl group consists of elements of the form
\[\pm a_{\sigma(1)}e_1 \pm \dots \pm a_{\sigma(n)}e_n.\]
Defining $i \le n$ to be maximal so $a_i$ is nonzero, we may bound $r(\beta, 2e_1)$ and $r(\beta, e_1 - e_2)$ in terms of $i$ and $n$ as in the $B_n$ case. For $n \ge 3$, this just leaves the case $i =1$  and the case $n = i = 3$. We handled the case $i  = 1$, $a_1 = 1$ above. The case $i = 1$,  $a_1 > 1$ can be handled by noting that $(a_1 - 1) e_1 +  e_2$ is also in $\Lambda(V)$. Finally, the case $n = i = 3$ can be handled by noting that  $a_1e_1 + (a_1 -1)(e_2 + e_3) $ is also in $\Lambda(V)$.

\textbf{Case: $\Phi \cong D_n$ with $n \ge 4$\quad}
We may write $\beta$ in the form
\[\tfrac{1}{2}(a_1e_1 + \dots + a_ne_n)\]
where the $a_i$ are integers of the same parity and $a_1 \ge \dots \ge a_{n-1} \ge |a_n|$. The orbit of this weight under the Weyl group consists of the weights
\[\tfrac{1}{2}\left(\pm a_{\sigma(1)}e_1 \pm \dots \pm a_{\sigma(n)}e_n\right),\]
where the number of negative signs in this expression is even. Taking $i$ to be maximal so $a_i$ is nonzero, there are at least  
\[\binom{n-2}{n-i} 2^{\min(n-2, i - 1)} + \binom{n-2}{n-i-1} 2^{i+1}\]
weights in this orbit not orthogonal to $e_1 - e_2$.  This is at least $5$ unless $i =1$ or $n = 4$ and $i = 4$. The former case can be handled as in the previous examples by subtracting $e_1 - e_2$ from $\beta$ when $a_1 > 2$, with $a_1 = 2$ being the standard representation.

In the latter case, we have already handled the case of $a_1 = 1$. The case of $a_1 > 1$ may be handled by noting that either $\beta - e_3 - e_4$ or $\beta - e_3 + e_4$ lies in $\Lambda(V)$.

This was the final Lie algebra to consider, and the theorem is shown. \qed

\subsection{Passing from $\C$ to $\QQ$}
\label{ssec:Q}
The descent to $\QQ$ requires one lemma for the orthogonal case.

\begin{lem}
\label{lem:quadform_descent}
Take $L/K$ to be an extension of number fields, choose a positive integer $n$, and take $\mfh$ to be a Lie subalgebra of $\mathfrak{sl}(K^n)$ such that $\mfh \otimes_K L$ equals $\mathfrak{so}_Q(L^n)$ for some nondegenerate quadratic form $Q$ on $L^n$. Then there is some nondegenerate quadratic form $P$ on $K^n$ such that $\mfh = \mathfrak{so}_{P} (K^n)$.
\end{lem}
\begin{proof}
We may assume that $L/K$ is Galois. Consider the vector space of symmetric $n \times n$ matrices $M$ with coefficients in $L$ such that $Mx + x^{\top} M = 0$ for all $x \in \mfh \otimes L$. This space is one-dimensional over $L$ and closed under the Galois action of $\Gal(L/K)$. By Hilbert 90, this space is generated by a matrix with coefficients in $K$.
\end{proof}

\begin{proof}[Proof of Theorem \ref{thm:semisimple}]

Suppose $V$ is a direct sum $V_1 \oplus V_2$, where $V_1, V_2$ are rational subspaces closed under the action of $\mfh$. The assumptions on $w$ imply that it acts trivially on either $V_1 \otimes \R$ or on $V_2 \otimes \R$. By the definition of the Lie subalgebra $\mfh$, we find that either $V_1$ or $V_2$ lies in the kernel of $\mfh$. By the assumptions of the theorem, either $V_1$ or $V_2$ is $0$. So $V$ is irreducible.

Furthermore, if we write $\mfh \otimes_{\QQ} \C$ in the form $\mfh_1 \oplus \dots \oplus \mfh_k$ with the $\mfh_i$ simple, we find that $w$ has zero projection to $\mfh_i$ for all but at most two $i \le k$ by Theorem \ref{thm:ss_complex}. Suppose first that it has one nonzero coordinate in this decomposition, say in $\mfh_1$.

 By the minimality assumption for $\mfh$, we find that $\mfh$ is simple. Take $K$ to be the center of the subring of the ring of vector space endomorphisms $\textup{End}_{\QQ} \,\mfh$ generated by the adjoint action of $\mfh$. Then $\mfh$ is absolutely simple as a Lie algebra over $K$ \cite{Jaco37}.

Considered as a $\mfh_1$ representation, $V \otimes_{\QQ} \C$ must equal a nontrivial irreducible representation summed with a number of trivial representations. By considering the Galois action on the $\mfh_i$, we see this is only possible if it takes the form
\[V_1 \oplus \dots \oplus V_k\]
where $V_i$ is nontrivial as an $\mfh_i$ module but is trivial as an $\mfh_j$ module for all $i \ne j$.

We have an isomorphism
\[\mfh \otimes_{\QQ} K \isoarrow \mfh \oplus \mfh'\]
of Lie algebras over $K$, where the map to $\mfh$ is the natural projection. For some embedding $K \hookrightarrow \C$, we may identify $\mfh_1$ with $\mfh \otimes_K \C$.  This embedding must be real, as $w$ would otherwise have nonzero image in the $\mfh_i$ corresponding to the conjugate embedding. 

The image of $\mfh$  in $V \otimes_{\QQ} K$ is then some representation $V_0$ for $\mfh$ over $K$ such that $V_0 \otimes_K \C$ is identified with $V_1$. We may choose $c$ in $K$ so that the map $V \to V \otimes_{\Q} K$ given by $v \mapsto v \otimes c$ then projects to give a nonzero map of $\mfh$ representations $V \to V_0$ over $\Q$. But $V$ and $V_0$ have the same dimension over $\Q$ and $V$ is irreducible, so $V$ and $V_0$ are isomorphic. In this way, we give $V$ the structure of an $\mfh$ representation over $K$.

With this done,  we see that $\mfh \otimes_K \C$ is isomorphic to either $\sl (V) \otimes_K \C$ or $\mathfrak{so} (V) \otimes_K \C$. In the former case, $\mfh$ must be $\sl( K^{n/d})$ since they have the same dimension. In the latter, we may apply Lemma \ref{lem:quadform_descent} to show it has the form $\mathfrak{so}_Q( K^{n/d})$.

This handles the case where $w$ projects to $0$ in $\mfh_i$ for all but one $i$. In the remaining case, we may suppose that $w$ has nonzero component in $\mfh_1$ and $\mfh_2$.

The $\mfh_i$ are all isomorphic to $\sl_2 (\C)$, and they are permuted by the absolute Galois group of $\QQ$ once we choose an embedding of an algebraic closure of $\QQ$ in $\C$. Take $G$ to be the image of this absolute Galois group in $S_k$, take $H$ to be the subgroup of $G$ fixing $\{1, 2\}$, and take $H_0$ to be the subgroup of $H$ fixing $1$.  Since $V$ is irreducible, we have an isomorphism
\[V \otimes \C  \cong \bigoplus_{\sigma \in G/H} V_{\sigma(1)} \otimes V_{\sigma(2)},\]
where $V_i$ is the standard representation of $\mfh_i$ viewed as an $\mfh \otimes \C$ module.  For $w$ to not have rank greater than $2$, $\{\sigma(1), \sigma(2)\}$ cannot meet $\{1, 2\}$ unless $\sigma$ represents the identity in $G/H$, so $G/H$ has size $k/2$. Take $F$ to be the extension of $\QQ$ associated with $H$, so $[F:\QQ] = k/2$, and take $K$ to be the extension associated with $H_0$, so either $K = F$ or $K/F$ is quadratic. We find that $F$ must have a real embedding.

If $K/F$ is quadratic, the same argument as before shows that $\mfh$ has the structure of a Lie algebra over $K$. If $K = F$, we see that $\mfh$ is the sum of two simple ideals, and that both are Lie algebras over $F$. In either case, by considering the image of the subalgebra $\mfh \otimes_F K \subseteq \mfh \otimes_{\Q} K$ acting on $V \otimes_{\QQ} K$, we find that $V$ is a representation of $\mfh$ over $F$. There is then some number field $L/K$ so $(\mfh \otimes_F L, V \otimes_F L)$ may be identified with $(\so_Q (L^4), L^4)$ for some nondegenerate quadratic form $Q$, as this is true for $L = \C$.  The result now follows from Lemma \ref{lem:quadform_descent}.
\end{proof}

\section{The proof of Theorem \ref{thm:unRatnered}: handling the unipotent radical}
\label{sec:radical}
Having proved Theorem \ref{thm:semisimple}, it is straightforward to prove Theorem \ref{thm:unRatnered} in the case that $H(\R)^0$ is semisimple. Specifically, if $U$ does not preserve any nonzero rational quadratic form, then $H(\R)^0$ takes the form $\SL_k(K) \otimes \R$ for some $k \ge 1$ and some number field $K$. This is isomorphic to $\SL_k (\R)^{r_1} \oplus \SL_k (\C)^{r_2}$ for some nonnegative integers $r_1, r_2$ such that $r_1 + 2r_2$ equals the degree of $K$. If we take $Z$ to be the portion of $\R^{r_1} \oplus \C^{r_2}$ with at least one of its $r_1 + r_2$ coordinates zero, we see that $H(\R)^0$ is transitive on $\R^n \backslash Z$.

We now aim to prove this same result without the semisimplicity condition, where it will follow as a consequence of a classification result for the unipotent radical of $H(\R)$.

Take $U = \{u_t\,:\, t \in R\}$ to be a $1$-parameter unipotent subgroup of $\SL_n (\R)$ such that $u_1 - \text{Id}$ is nilpotent and either has rank $1$ or has rank $2$ with $(u_1 - \text{Id})^2$ nonzero. We will assume that $U$ preserves no nonzero rational quadratic form.

Take $H$ to be the minimal Zariski closed subgroup of $\SL_n(\Q)$ such that $H(\R)$ contains $U$, and take $\mfh$ to be the Lie algebra associated to $H$ over $\Q$. We take $w \in \mfh \otimes \R$ to be the nilpotent element such that $\exp(w) = u_1$. Then $\mfh$ may be characterized as the minimal rational Lie algebra such that $\mfh \otimes \R$ contains $w$.

We choose a Levi decomposition
\[\mfh = \mfs + \mfn\]
for $\mfh$, where $\mfn$ is the radical of $\mfh$ and $\mfs$ is semisimple. We note that the radical of $\mfh$ is nilpotent. Otherwise, we could define a nontrivial algebraic homomorphism $H(\C) \to \mathbb{G}_m^{\times}$, and this would contradict the minimality of $H$ since $U$ lies in the kernel of any such homomorphism.

Given a subspace $W$ of $\R^n$, we take $W^{\QQ}$ to be the least rational subspace $W'$ of $\Q^n$ such that $W' \otimes \R$ contains $W$, and we take $W_{\QQ}$ to be the greatest rational subspace contained in $W$. Since $U$ preserves no nonzero rational quadratic form, we have
\[\Q^n = \Img(w)^{\Q},\]
as any quadratic form on $\Q^n/\Img(w)^{\Q}$ corresponds to a quadratic form preserved by $U$.

Take $r \le 2$ to be the rank of $w$. We then may choose subspaces $V_1, \dots, V_4$ of $\Q^n$ such that
\begin{alignat*}{2}
&V_4 = \ker(w)_{\Q} \cap \Img(w^r)^{\Q} && V_3 \oplus V_4 = \Img(w^r)^{\Q} \\
 &V_2 \oplus V_3 \oplus V_4 =   \Img(w^r)^{\Q} + \ker(w^r)_{\Q} \qquad && V_1 \oplus V_2 \oplus V_3 \oplus V_4 =  \Q^n.
\end{alignat*}
For $i, j \le 4$, take $\mfO_{ij} = \Hom(V_i, V_j)$. We see that $w$ lies in
\[\left(\mfO_{11} \oplus \mfO_{33} \oplus \bigoplus_{i < j} \mfO_{ij}\right) \otimes \R.\]
For example, given $v$ in $\ker(w^r)_{\Q}$, we see that $wv$ must be $0$ if $r =1$. If $w$ has rank $2$, we instead note that $\ker(w) \cap \Img(w)$ is one-dimensional, and hence must be $\Img(w^2)$. In either case, we find that $w \ker(w^r)_{\Q}$ lies in $\Img(w^r)$, so $w$ maps $V_2$ into $V_3 \oplus V_4$. The other entries are similar.

Furthermore, we see that the image of $w$ in $\mfO_{11} \otimes \R$ is an endomorphism of $V_1 \otimes \R$ whose kernel has zero intersection with $V_1$. After all, given $v \in V_1$ in this kernel, we see that $\langle wv \rangle^{\Q} + \Img(w^2)^{\Q}$ cannot equal $\Img(w)^{\Q}$, so $wv$ must lie in $\Img(w^2)$, implying $v$ is in $\ker(w^2)_{\Q}$ and hence is $0$. So, if $V_1$ is nonzero, we find it must be an irreducible nontrivial $\mfh$ representation. A similar argument shows the same for $V_3$.

From these considerations, and since $\mfh$ is minimal among Lie algebras whose tensor product with $\R$ contains $w$, it follows that we may rechoose $V_1$ and $V_3$ so that
\begin{equation}
\label{eq:Levi_geography}
\mfs \subseteq \mfO_{11} \oplus \mfO_{33} \quad\text{and}\quad \mfn \subseteq \bigoplus_{i < j} \mfO_{ij}.
\end{equation}
We also note that an integral quadratic form $P$ on $\R^n$ is preserved by $U$ if and only if it is preserved by $H(\R)$.

The following proposition implies Theorem \ref{thm:unRatnered}.
\begin{prop}
\label{prop:all_AB}
Take all notation as above. Choose $a$ minimal and $b$ maximal so $V_a$ and $V_b$ are nonzero. Then, if $(a, b)$ is $(3, 4)$ or $(1, 4)$, then $\mfn$ contains $\mfO_{ab}$. Further, if $(a, b)$ is $(1, 3)$, then
\[(\mfn \cap \mfO_{13})v = V_3 \quad\text{for all }\, v\in V_1 \backslash 0.\]
\end{prop}

\begin{proof}[Proof of Theorem \ref{thm:unRatnered} assuming Proposition \ref{prop:all_AB}]
We prove a slightly stronger form of Theorem \ref{thm:unRatnered}, where we assume that $u_1 - \text{Id}$ is nilpotent and has rank at most $2$ and  $(u_1 - \text{Id})^2 \ne 0$ if its rank is exactly $2$.  

Take $a$ minimal so $V_a$ is nonzero. If $a = 2$, then \eqref{eq:Levi_geography} shows $\Q^n/V_3 \oplus V_4$ is a trivial $H$ representation, and hence that $H$ preserves a nonzero quadratic form on $\Q^n$, giving the theorem in this case. So we may assume $a = 1$ or $a =3$. 

We then see that the image of $H$ in $\text{GL}\left(V_a\right)$ takes the form $\SL_k (K)$ for some number field $K$ so long as $U$ preserves no nonzero rational quadratic form. As above, we find there is some finite union $Z_a$ of proper subspaces of $V_a \otimes \R$ such that $H(\R)^0$ acts transitively on $V_a \otimes \R\backslash Z_a$. Taking $Z$ to be the preimage of $Z_a$ in $\Q^n$, we claim that $H(\R)^0$ acts transitively on $\R^n \backslash Z$ so long as $U$ does not preserve any nonzero rational quadratic form.

Suppose this result is known in $\R^m$ for all $m < n$, and consider an example in $\R^n$. Choose $b$ maximal so $V_b$ is nontrivial. If $a = b = 3$, then the result follows for this example by Theorem \ref{thm:semisimple}. So we may assume that $(a, b)$ is in one of the three cases of Proposition \ref{prop:all_AB}.

By our assumptions, $H(\R)^0$ acts transitively on
\[\big((\Q^n/V_b) \otimes \R\big) \big\backslash Z/(V_b \otimes \R).\]
Furthermore, by Proposition \ref{prop:all_AB}, we find that every fiber in the map
\[ \R^n \backslash Z \xrightarrow{\quad} \big((\Q^n/V_b) \otimes \R\big) \big\backslash Z/(V_b \otimes \R) \]
consists of points from a single orbit under the action of $H(\R)^0$. Together, these imply the action of $H(\R)^0$ on $\R^n \backslash Z$ is transitive.
\end{proof}

\subsection{The $\mfs$ representations $V_1$ and $V_3$ }
\label{ssec:V1V3}
From Theorem \ref{thm:semisimple}, we have a good understanding of the structure of $\mfs$. We can use this to classify the possible forms of the $\mfs$ representations $V_1$ and $V_3$, which then gives some useful constraints on the form that $\mfn$ can take.
\begin{lem}
\label{lem:ss13}
The triple $(\mfs, V_1, V_3)$ is compatibly isomorphic to a triple of one of the following forms:
\begin{alignat*}{2}
&(1)\quad  (0, 0, 0) &&  (2)\quad  \left(\so_Q(K^m), \, 0,\, K^m\right) \\
&(3)\quad  \left(\sl(K^m),\, 0, \, K^m \right) \qquad && (4)\quad\left(\sl(K^m), \,K^m,\,0\right) \\
& (5)\quad  \left(\sl(K^m), \, K^m,\, (K^m)^*\right) && (6) \quad \left(\sl( K^m),\,K^m,\, K^m\right)\, \text{ with }\, m \ge 3 \\ 
& (7) \quad \left( \sl(K^m) \oplus \sl(L^p), \,K^m, \, L^p\right).\qquad &&
\end{alignat*}
Here, $K$ and $L$ are number fields, $m, p \ge 2$ are integers, $Q$ is a nondegenerate quadratic form on $K^m$, and the representations of the Lie algebras are the standard ones except on $(K^m)^*$, which is endowed with the negative transpose representation of $\sl(K^m)$.
\end{lem}
\begin{proof}
Given $i = 1, 3$, if $V_i$ is nontrivial, then the image of $\mfs$ in $\sl(V_i)$ is isomorphic to either $\sl(K^m)$ or $\so_Q (K^m)$ for some choice of $K$, $m \ge 2$, and $Q$, with the orthogonal case only possible if $w$ has rank $2$ as an endomorphism of $V_i \otimes \R$. We see that this rank is $1$ unless $\Img(w)^{\QQ} = \Img(w^2)^{\QQ}$, in which case $V_2 = V_1 = 0$. So, outside the case (2) listed above, we find that these images must be of the form $\sl( K^m)$. It is then clear that the triple is of the form (1), (2), (3), or (4) if either $V_1$ or $V_3$ is zero.

Now suppose $V_1$ and $V_3$ are both nonzero.  Take $\mfs_i$ to be the image of $\mfs$ in $\sl(V_i)$; this is a simple Lie algebra. The inclusion
\[\mfs \hookrightarrow \sl( V_1) \oplus \sl (V_3)\]
then is either a surjection onto $\mfs_1 \oplus \mfs_3$ or takes the form of a graph of an isomorphism of Lie algebras $\mfs_1 \isoarrow \mfs_3$. In the former case, the triple must be of the form (7).

Finally, in the last case, we have $\mfs \cong \sl( K^m)$ for some $m \ge 2$. Here, $K$ may be taken to be the center of the subring of $\text{End}_{\Q}\, \mfs$ generated by the image of $\mfs$ under the adjoint action. With this definition, $V_1$ and $V_3$ are representations of $\mfs$ over $K$.  Choose identifications $V_1, V_3 \cong K^m$, and take $\varphi_i: \mfs \to \sl( K^m)$ to be the representation associated to $V_i$. Then there is an automorphism $\iota$ of the Lie algebra $\sl( K^m)$ over $K$ such that $\varphi_3 = \iota \circ \varphi_1$. The automorphism $\iota \otimes_K \C$ of $\sl( K^m) \otimes_K \C$ is, up to inner automorphism, equal to either the identity map or the negative transpose map. An element of $\text{GL}_m (\C)$ is determined up to scalar multiple by its action on $\sl_m (\C)$, so we find from Hilbert 90 that every inner automorphism of $\sl( K^m) \otimes \C$ that preserves $\sl(K^m)$ is given by an inner automorphism of $\sl(K^m)$. So $V_1$ is isomorphic to either $V_3$ or $V_3^*$, and is isomorphic to both if $m = 2$. So we find that the triple must take the form (5) or (6) in these cases.
\end{proof}
We apply this lemma to study the structure of the representation $\mfO_{13}$, which is nontrivial in cases (5), (6), and (7). First suppose we are in case (7). Write 
\[\sl( K^m) \otimes_{\Q} \C = \mfs_1 \oplus \dots \oplus \mfs_d\quad\text{and}\quad \sl( L^p)\otimes_{\Q} \C = \mfs'_1\oplus \dots \oplus \mfs'_e,\]
where the $\mfs_i$ and $\mfs'_e$ are all simple, so $d = [K: \QQ]$ and $e = [L:\QQ]$. Taking $V_{1i} = \mfs_i V_1$ and $V_{3j} = \mfs'_j V_3$ for $i \le d$ and $j \le e$, $\mfO_{13} \otimes \C$ is a sum of representations of the form
\[\Hom(V_{1i}, V_{3j})\quad\text{with }\, i, j \le d.\]
These representations are nontrivial, nonisomorphic, and also not isomorphic to any $V_{1k}$, $V_{1k}^*$, $V_{3k}$, or $V_{3k}^*$.

The Galois action of $\Gal(\C/\Q)$ defines permutation actions on the components of $\mfs \otimes \C$. These give transitive permutation actions on $[d]$ and $[e]$. For each orbit $S$ of the Galois action on $[d] \times [e]$, there is a subrepresentation $W_S$ of $\mfO_{13}$ such that
\[W_S \otimes \C = \bigoplus_{(i, j)} \Hom(V_{1i}, V_{3j}).\]
These $W_S$ are distinct irreducible representations, and they are nontrivial and not isomorphic to $V_1$, $V_1^*$, $V_3$, or $V_3^*$.

In case (5) and (6), we decompose $\sl(K^m) \otimes_{\Q} \C$ as $\mfs_1 \oplus \dots \oplus \mfs_d$ as before, and we still take $V_{1i} = \mfs_i V_1$.

In case (6), we identify $\mfO_{13}$ with $\Hom_{\Q}(V_1, V_1)$. The scalar multiplications in this module give a copy of $K \otimes_{\Q} \C$ in $\mfO_{13} \otimes \C$. The remaining irreducible components in this representation take the form
\[\text{End}_0(V_{1i}) \quad\text{and}\quad \Hom(V_{1i}, V_{1j}), \quad i, j \le d,\,\,\, i \ne j.\]
These are again nonisomorphic, nontrivial, and distinct from any $V_{1i}$ or $V_{1i}^*$.

The sum of the $\text{End}_0(V_{1i})$ equals $\text{End}_{0, K}(V_1) \otimes_{\Q} \C$. With this and the copy of $K$ removed, the other irreducible components of $\mfO_{13}$ take the form $W_S$ for a Galois orbit $S$ of $[d] \times [d]$ not containing $(1, 1)$, where $W_S \otimes \C$ equals $\bigoplus_{(i, j) \in S} \Hom(V_i, V_j)$. These $W_S$ are distinct, nontrivial, and nonisomorphic to $V_1$ or $V_1^*$.

In case (5), we choose some isomorphism $V_1^* \to V_3$ and use it to define an identification of $\mfO_{13}$ with $\Hom(V_1, V_1^*)$, which is the sum of $\Hom_{\text{sym}}(V_1, V_1^*)$ and $\Hom_{\text{alt}}(V_1, V_1^*)$. Inside $\Hom_{\text{alt}}(V_1, V_1^*)$, we have the set of $K$-equivariant maps $\Hom_{K, \text{alt}}(V_1, V_1^*)$, which satisfies
\[\Hom_{K, \text{alt}}(V_1, V_1^*) \otimes \C = \bigoplus_{i \le d} \Hom_{\text{alt}}\left(V_{1i}, V_{1i}^*\right).\]
The representation $\Hom_{K, \text{alt}}(V_1, V_1^*)$ is a nontrivial representation unless $m  = 2$, where it is isomorphic to $K = \Q^d$. It is never isomorphic to $V_1$ or $V_1^*$.

The irreducible subrepresentations of $\Hom_{\text{sym}}(V_1, V_1^*) \otimes \C$ take the form
\[A_{ij} = \{x + x^{\top}\,:\,\, x \in \Hom(V_{1i}, V_{1j}^*)\}\quad\text{with }\, i \le j.\]
These are distinct nontrivial representations, and they are not isomorphic to any $V_{1k}$ or $V_{1k}^*$. From the Galois action, the irreducible representations of $\Hom_{\text{sym}}(V_1, V_1^*)$ take the form $W_S$ with
\[W_S \otimes \C = \bigoplus_{\substack{(i, j) \in S \\ i \le j}} A_{ij},\]
where $S$ takes the form $T \cup T^*$ with $T$ a Galois orbit on $[d] \times [d]$ and $T^*$ defined by $\{(j, i)\,:\,\, (i, j) \in S\}$. These representations are distinct, nontrivial, and not isomorphic to $V_1$, $V_1^*$, or $\Hom_{K, \text{alt}}(V_1, V_1^*)$.

\begin{lem}
\label{lem:min13}
Suppose we are in case (5), (6), or (7) of Lemma \ref{lem:ss13}. Then $\Img(\mfn \to \mfO_{13})$ contains   $\Hom_{K, \textup{alt}}(V_1, V_1^*)$ in case (5),  $\textup{End}_{0, K}(V_1)$ in case (6), and some $W_S$ in case (7).
\end{lem}
The proof of this lemma will use a simple observation that will recur later. Specifically, we have
\begin{equation}
\label{eq:34_surj}
\Img(\mfn \to \mfO_{12}) = \mfO_{12}\quad\text{and}\quad \Img(\mfn \to \mfO_{34}) = \mfO_{34}
\end{equation}
To prove the first of these, note that this image must take the form $\Hom(V_1, V_2')$ for some subspace $V_2'$ of $V_2$ since $V_1$ is irreducible as an $\mfs$ representation if it is nonzero. We have $\langle wv\rangle^{\Q} + \Img(w^r)^{\Q} = \Img(w)^{\Q}$ for all nonzero $v \in V_1$ since such $v$ cannot lie in the kernel of $w^2$, and this implies that $\mfn v$ has image $V_2$ in $V_2$, implying $V_2' = V_2$.

For the second identity, we note that $\Img(\mfn \to \mfO_{34})$ must take the form $\Hom(V_3, V_4')$ for some subspace $V_4'$ of $V_4$, and the claims follows since we assumed that $V_4$ is contained in $\Img(w^r)^{\Q}$.

\begin{proof}[Proof of Lemma \ref{lem:min13}]
Write the projection of $w$ to $\mfO_{ij} \otimes \C$ as $w_{ij}$. Since $w^2$ has nonzero component in $\mfO_{13} \otimes \C$, either $w_{23} \circ w_{12}$ is nonzero or $w_{13} \circ w_{11} + w_{33} \circ w_{13}$ is nonzero.

In case (5) and (6), write $\mfs \otimes \C$ in the form $\mfs_1 \oplus \dots \oplus \mfs_d$, and use this decomposition to define $V_{1i}$. From our discussion in Section \ref{ssec:Q}, we may assume that $w$ has nonzero component only in $\mfs_1$. Then $w_{11}$ lies in $\Hom(V_{11}, V_{11})$, with $w_{33}$ then lying in $\Hom(V_{11}^*, V_{11}^*)$ in case (5) and $\Hom(V_{11}, V_{11})$ in case (6). Further $w_{12}$ lies in $\Hom(V_{11}, V_2 \otimes \C)$, $w_{23}$ lies in $\Hom(V_{2} \otimes \C, V_{11}^*)$ or $\Hom(V_{2} \otimes \C, V_{11})$, and $w_{13}$ lies in
\[\Hom(V_{11}, V_1^* \otimes \C) + \Hom(V_{1} \otimes \C, V_{11}^*) \quad\text{or}\quad\Hom(V_{11}, V_1 \otimes \C) + \Hom(V_{1} \otimes \C, V_{11}).\]

First suppose $w_{23} \circ w_{12} \ne 0$ in any of the three cases. If  the map from
\[\Img(\mfn \to \mfO_{12} \oplus \mfO_{23})\]
to $\mfO_{12}$ is not injective, then \eqref{eq:34_surj} and the closure of $\mfn$ under Lie brackets implies that $\mfn$ surjects onto $\mfO_{13}$. So we may assume this is the graph of some homomorphism $\kappa: \mfO_{12} \to \mfO_{23}$. In cases (6) and (7), this homomorphism must be zero, so we must have $w_{23} \circ w_{12} = 0$ in these cases.

So suppose we are in case (5).  There is then nonzero $x\in V_{11}^*$, $v_2 \in V_2 \otimes \C$, and $v_2'\in V_2' \otimes \C$ so that $w_{12} = v_2 \otimes x$ and $w_{23} = x \otimes v_2'$, where $v_2'(v_2) \ne 0$. The Lie brackets in $[[\mfs_1, w], w]$ then generate $\Hom_{\text{alt}}(V_{11}, V_{11}^*)$, so $\mfn$ contains $\Hom_{K, \text{alt}}(V_1, V_1^*)$.

So we now may assume $w_{13} \circ w_{11} + w_{33} \circ w_{13}$ is nonzero. This gives the result immediately in case (7), since $\Img(\mfn \to \mfO_{13})$ is nonzero.

In case (5) and (6), we know from our discussion in Section \ref{ssec:Q} that $w_{11}$ lies in $\Hom(V_{11}, V_{11})$, with $w_{33}$ then lying in $\Hom(V_{11}^*, V_{11}^*)$ in case (5) and $\Hom(V_{11}, V_{11})$ in case (6). From our assumption, we find that $w_{13}$ has nonzero component in $\Hom(V_{11}, V^*_{11})$ in case (5) and $\Hom(V_{11}, V_{11})$ in case (6). For case (6), we see since $m\ge 3$ that the component of $w_{13}$ in $\Hom(V_{11}, V_{11})$ cannot lie in the center, and the result follows in this case.

Finally, in case (5), the assumption $w_{13} \circ w_{11} + w_{33} \circ w_{13} \ne 0$ implies that $w_{13}$ projects nontrivially to $\Hom_{\text{alt}}(V_{11}, V_{11}^*)$. It also projects trivially to every other $\Hom_{\text{alt}}(V_{1i}, V_{1i}^*)$. Since $\mfn \otimes \C$ is closed under the Galois action, the result follows.

\end{proof}

\subsection{The proof of Proposition \ref{prop:all_AB}}
We may assume without loss of generality that  $\dim V_4 \le 1$. To see this, suppose $V_4$ has dimension at least $2$. Since $V_a$ is an irreducible $\mfs$ representation, we see that $\mfn \cap \mfO_{a4}$ is $\Hom(V_a, V_4')$ for some subspace $V_4'$ of $V_4$. If $V_4' = V_4$, the result follows. Otherwise, we note that the result follows if it holds for all spaces of the form $V/V_4''$, where $V_4''$ is a codimension $1$ subspace of $V_4$ not contained in $V_4'$. This gives the reduction.

By considering $\mfn$ as an $\mfs$ representation, and in particular considering the irreducible components of $\mfO_{13}$ using our work in Section \ref{ssec:V1V3} above, we have a decomposition
\begin{align}
\label{eq:checkerboard}
\mfn \,=\,& \Img\left(\mfn \to \mfO_{12} \oplus  \mfO_{14} \oplus \mfO_{23}\oplus \mfO_{34}\right)  \oplus \Img\left(\mfn \to  \mfO_{13} \oplus \mfO_{24} \right).
\end{align}

We now divide our proof based on the cases of $(\mfs, V_1, V_3)$ in Lemma \ref{lem:ss13}.

\subsubsection{Cases (1) and (2)}
In either of these cases, the conditions of the proposition are not met. In case (1), we see that $a$ cannot be either $1$ or $3$. In case (2), we instead find that $H$ preserves a nonzero rational quadratic form. Indeed, if $\mfs$ is isomorphic to $\mathfrak{so}_Q(K^m)$, then $\text{tr}\circ (c\cdot Q)$ is preserved by $H$ for any $c$ in $K$.

\subsubsection{Case (3)}
In this case, we may assume $a = 3$ and $b = 4$, so  $\mfn$ is a subspace of $\mfO_{34}$. The result follows from \eqref{eq:34_surj}.

\subsubsection{Case (4)}
In this case, we may assume $a = 1$ and $b = 4$. By the definition of $V_2$ and $V_4$, we see that $\Img(\mfn  \to \mfO_{24})$ is nonzero if $V_2$ is nonzero. By \eqref{eq:checkerboard} and \eqref{eq:34_surj}, we then find that $[\mfn, \mfn]$ contains $\mfO_{14}$ if $V_2$ is nonzero. If $V_2$ is $0$, then $\mfn$ is contained in $\mfO_{14}$, and the result follows since $V_4$ equals $\Img(w^2)^{\Q}$.

\subsubsection{Case (5)}
We now move on to the most difficult case. In the remaining cases, $a$ is $1$, and $b$ is either $3$ or $4$. We handle the case where $b = 4$ first. 

 If $\mfn \cap \mfO_{13}$ is nonzero, then \eqref{eq:34_surj} implies that $[\mfn, \mfn \cap \mfO_{13}]$ contains $\mfO_{14}$. So we may assume that this intersection is $0$. By Lemma \ref{lem:min13},  $\Img(\mfn \to \mfO_{13})$ equals $\Hom_{K, \text{alt}}(V_1, V_1^*)$, and $m = 2$.

Given $m$ in $\mfn$, we denote the image of $m$ in $\mfO_{ij}$ by $s_{ij}(m)$. 

From \eqref{eq:checkerboard} and \eqref{eq:34_surj}, we now find that
\[\big[ \mfn \cap (\mfO_{13} \oplus \mfO_{24}),\, \mfn \cap  (\mfO_{12} \oplus \mfO_{14} \oplus \mfO_{23} \oplus \mfO_{34})\big]\]
contains $\mfO_{14}$ unless there is an isomorphism $\Gamma: \mfO_{12} \to \mfO_{34}$ of $\mfs$ representations and an injection $\iota: \Hom_{K, \text{alt}}(V_1, V_1^*)  \hookrightarrow \mfO_{24}$ such that, for $x \in \mfn$, we have
\[s_{34}(x) = \Gamma(s_{12}(x)) \quad\text{and}\quad \iota(s_{13}(x)) = s_{24}(x),\]
with
\[s_{34}(x) \circ s_{13}(x) - s_{24}(x) \circ s_{12}(x) = 0.\]
This is possible only if $V_2$ is one-dimensional since $V_4$ is one-dimensional. This in turn allows us to assume that $K = \Q$.

Since $s_{24}$ is nonzero, we see that no nonzero element in $V_2$ is in $\ker(w)$, so $\Img(\mfn \to \mfO_{23}) = \mfO_{23}$. Since $\Img([\mfn, \mfn] \to \mfO_{13})$ is contained in $\Hom_{\text{alt}}(V_1, V_1^*)$, we find that $\Img(\mfn \to \mfO_{12} \oplus \mfO_{23})$ is the graph of an isomorphism $\mfO_{12} \to \mfO_{23}$. We similarly find that $\Img(\mfn \to \mfO_{23} \oplus \mfO_{34})$ is the graph of some isomorphism.

After adjusting the bases for $V_1$ and $V_3$ if necessary, we find that there is some nonzero $k_1, k_2, k_3$ such that every $x$ in $\mfh$ takes the form
\[\left(\begin{matrix}
0 & -k_1e & k_1d & k_3f & - & - \\
0 & a & b & d & 0 & f \\
0 & c & -a & e & -f & 0\\
0 & 0 & 0 & 0 & k_2d & k_2 e \\
0 & 0 & 0 & 0 & -a & - c \\
0 & 0 & 0 & 0 & -b &  a 
\end{matrix}\right)\]
for some rational numbers $a, b, c, d, e, f$; and we find that, given any value of these rational numbers, some element in $\mfh$ takes this form. Choose $x$ in $\mfn$ with $d = 1$ and $e = f = 0$, and choose $y$ in $\mfn$ with $e = 1$ and $d = f = 0$. Then
\[[x, [x, y]] =  \left(\begin{matrix}
0 & 0 & 0 & 0 & -3k_1k_2 & 0 \\
0 & 0 & 0 & 0 & 0 & 0 \\
0 & 0 & 0 & 0 & 0 & 0\\
0 & 0 & 0 & 0 & 0 & 0 \\
0 & 0 & 0 & 0 & 0 & 0 \\
0 & 0 & 0 & 0 & 0 &  0 
\end{matrix}\right),\]
so $\mfn \cap \mfO_{14}$ is nontrivial, giving the case if $b = 4$.

So suppose $b = 3$. Choosing some nonzero $v_1 \in V_1$, we suppose $V = (\mfO_{13} \cap \mfn)v_1$ is a proper subspace of $V_3$. In this case, we wish to prove $H$ preserves some nonzero rational quadratic form. Choose a nontrivial homomorphism $y: V_3/V \to \Q$.

 We may identify $V_3$ with either $\Hom_K(V_1, K)$ and $\Hom_{\Q}(V_1, \Q)$. Under the former identification, we know that $V$ contains all $v' \in \Hom_K(V_1, K)$ such that $v'(v) = 0$ by Lemma \ref{lem:min13}, so $y$ takes the form 
\[v' \mapsto \text{tr}_{K/\Q}(c v'(v))\]
 for some $c$ in $K$. So we may choose some isomorphism 
\[\phi_{31}: V_3 \to V_1^* = \Hom_{\Q}(V_1, \Q)\]
 such that $y(v_3) = \phi_{31}(v_3)(v_1)$ for all $v_3$ in $V_3$.

Applying $\phi_{31}$ to $\Img(\mfn \to \mfO_{13})$ gives a subspace of $\Hom(V_1, V_1^*)$. If this subspace contains any subrepresentation of $\Hom_{\text{sym}}(V_1, V_1^*)$, then our discussion before Lemma \ref{lem:min13} shows that $y$ must attain arbitrary rational values on $V$. So we find that it lies in $\Hom_{\text{alt}}(V_1, V_1^*)$.

So
\[\phi_{31} \circ s_{23}(x) \circ s_{12}(y) - \phi_{31} \circ s_{23}(y) \circ s_{12}(x) \in \Hom_{\text{alt}}(V_1, V_1^*)\]
for all $x, y$ in $\mfn$. This and \eqref{eq:34_surj} imply that $\Img(\mfn \to \mfO_{12} \oplus \mfO_{23})$ is the graph of some isomorphism $\Gamma: \mfO_{12} \to \mfO_{23}$. 

Choose a basis $e_1, \dots, e_k$ for $V_2$. We may view $\phi_{31} \circ \Gamma$ as a matrix $M$ with coefficients in $K$. This matrix satisfies
\begin{equation}
\label{eq:alt135}
(Mx)  \circ y  - (My) \circ x  \in \Hom_{\text{alt}}(V_1, V_1^*)\quad\text{for }\, x, y \in \mfO_{12}.
\end{equation}

If we take $x$ in $\Hom(V_1, \langle e_1 \rangle)$ and $y$ in $\Hom(V_2, \langle e_2 \rangle)$, we find from \eqref{eq:alt135} that  the coefficients $M_{12}$ and $M_{21}$ are equal. By repeating this argument, we find that the matrix is symmetric.

Take $c = M_{11}$. Then, taking $x$ in $\Hom(V_1, \langle e_1 \rangle)$ and $y = cx$, we have
\[(Mx) \circ cx - (Mcx) \circ x =  c^2x \otimes_{\Q} x - cx \otimes_{\Q} cx.\]
Since this needs to be alternating, we find that $c$ lies in $\Q$. Repeating this argument for other choices of bases, we find that $\phi_{31} \circ \Gamma$ corresponds to a matrix with rational entries. 

In other words, there is a symmetric map $\phi_{22}: V_2 \to V_2^*$ such that $\Gamma$ is 
\[s_{23}(x)^{\top} \circ \phi_{31}^{\top}  =\phi_{22} \circ  s_{12}(x) \quad\text{for all } x \in \mfn.\]
But now we can calculate that 
\[\phi := \phi_{31} - \phi_{22} + \phi_{31}^{\top}: \Q^n \to \left(\Q^n\right)^*\]
corresponds to a nonzero rational quadratic form preserved by $H$, as
\[\phi(hv)(v) = 0 \quad\text{for all } h \in \mfh \text{ and }\, v \in \Q^n.\]
This finishes the case.

\subsubsection{Cases (6) and (7)}
In these cases, $(\mfn \cap \mfO_{13})v_1 = V_3$ for all $v_1 \in V_1 \backslash 0$ by Lemma \ref{lem:min13}, as $v_1$ projects nontrivially to all $V_{1i}$ as defined in that lemma. This gives the result if $b = 3$. If $b = 4$, this lemma, \eqref{eq:34_surj}, and \eqref{eq:checkerboard} imply that
\[[\mfn \cap \mfO_{13}, \,\mfn] \supseteq \mfO_{14},\]
giving the result in this case. These were the last cases to consider. \qed

\bibliography{references}{}
\bibliographystyle{amsplain}

\end{document}